\documentclass[12pt,a4paper,final]{article}
\RequirePackage[l2tabu, orthodox]{nag} 
\usepackage{geometry}
\usepackage[english, strings]{babel}
\usepackage{nchairx}
\usepackage[utf8]{inputenc}
\usepackage[T1]{fontenc}       
\usepackage{longtable}         
\usepackage{exscale}           
\usepackage[sort]{cite}        
\usepackage{xspace}            
\usepackage{tikz}              
\usetikzlibrary{cd, decorations.pathmorphing}
\usepackage{ifdraft}           
\usepackage{appendix}
\usepackage[expansion=false    
]{microtype}        
\usepackage{imakeidx} 
\usepackage[backref=page,      
final=true,         
pdfpagelabels,       
]{hyperref}         
\mathtoolsset{showonlyrefs,showmanualtags}
\usepackage{parskip}
\usepackage{authblk} 

\usepackage[numbers,sort]{natbib}

\makeatletter
\renewcommand{\NAT@separator}{\NAT@sep\nolinebreak}
\makeatother

\makeatletter

\newcommand{\bibnote}[2]{\nocite{#1}\@namedef{#1chairxnote}{#2}}
\makeatother

\renewcommand{\tilde}{\widetilde}

\newcommand{\C}{\field{C}}

\newcommand{\R}{\field{R}}

\newcommand{\N}{\field{N}}

\newcommand{\Z}{\field{Z}}

\renewcommand{\epsilon}{\varepsilon}

\DeclareFontFamily{U}{mathx}{}
\DeclareFontShape{U}{mathx}{m}{n}{ <-> mathx10 }{}
\DeclareSymbolFont{mathx}{U}{mathx}{m}{n}
\DeclareFontSubstitution{U}{mathx}{m}{n}
\DeclareMathAccent{\widecheck}{0}{mathx}{"71}

\renewcommand{\Holomorphic}{\mathscr{H}}

\newcommand{\GL}{\operatorname{GL}}

\newcommand{\liegroup}[1]{\operatorname{#1}}

\newcommand{\Taylor}{\operatorname{T}}

\newcommand{\Majorant}{\operatorname{F}}

\newcommand{\LieAlg}{\operatorname{Lie}}

\newcommand{\Entire}{\mathscr{E}}

\newcommand{\Covering}[1][G]{\widetilde{#1}} 

\author{\sc Michael Heins\thanks{This paper has been supported from Italian
Ministerial grant PRIN 2022 ``Cluster algebras and Poisson Lie groups'', n.
20223FEA2E - CUP D53C24003330006}}
\affil{Università degli Studi di Salerno}

\title{\sc Entire Functions on Lie Groups}
\date{December 2025}

\footnotetext{\textbf{MSC classication:} 22E30, 30H50, 30B40}

\begin{document}

\maketitle

\begin{abstract}
    Every Lie group $G$ carries a distinguished algebra of particularly well-behaved
    real-analytic mappings: The entire functions $\Entire(G)$. They were introduced
    for the purposes of strict deformation quantization. This paper establishes a
    one-to-one correspondence between entire functions and holomorphic mappings
    $\Holomorphic(G_\C)$ on the universal complexification $G_\C$ of $G$ as Fréchet
    algebras. Methodically, this is achieved by porting aspects of classical complex
    analysis into a left-invariant guise and by studying the geometry of $G_\C$. As a
    byproduct, we obtain a strict deformation quantization of the holomorphic
    cotangent bundle of any universal complexification.
\end{abstract}

\vspace{1cm}

\tableofcontents

\vspace{1.5cm}

\textbf{Acknowledgments:} The author would like to thank Chiara Esposito, Stefan
Waldmann and Oliver Roth for numerous valuable discussions and remarks, which
have improved the presentation and clarity of this article significantly.

\newpage

\section{Lie-Taylor Formulas}
\label{sec:LieTaylor}%
This article is partially based on and expands upon the results obtained within the
PhD thesis \cite[Ch.~2]{heins:2024a}. Moreover, it constitutes a follow-up to
\cite{heins.roth.waldmann:2023a}, completely clarifying the nature of its central
definition. That being said, our methods are elementary throughout and the
presentation is entirely self-contained. That is to say, this article supercedes and
incorporates all the relevant parts of the prior material. We begin by setting the
stage in terms of definitions and precise formulations of our main results.

Throughout the text we deal with a, typically real, Lie group $G$ of dimension
\begin{equation}
    d
    \in
    \N_0
    \coloneqq
    \{0\}
    \cup
    \N
    \coloneqq
    \{0\}
    \cup
    \{1,2,\ldots\}
\end{equation}
with Lie algebra $\liealg{g}$ and unit $\E$. Sometimes, the functorial notation
$\LieAlg(G)$ for the Lie algebra is more convenient, and we use both
interchangeably throughout. Whenever complex Lie groups occur, we shall
forget about the complex structure to apply the real constructions regardless. If $X$
is a metric or normed space, we denote the corresponding open balls of radius~$r
\ge 0$ centered at~$x_0 \in X$ by $\Ball_r(x_0)$ and their closures by
$\Ball_r(x_0)^\cl$. Finally, a domain within a topological space is an open, non-empty
and connected subset.

In \cite{heins.roth.waldmann:2023a}, the algebra $\Entire(G)$ of entire functions on
any Lie group $G$ was introduced to facilitate a strict deformation quantization of
the cotangent bundle $T^*G$. Entire functions constitute real-analytic functions,
whose Lie-Taylor series is particularly well-behaved. However, concrete examples
have been sparse. Indeed, the only provided construction consists in the
representative functions or matrix elements by
\cite[Theorem~4.23]{heins.roth.waldmann:2023a}. These arise from finite dimensional
representations
\begin{equation}
    \pi
    \colon
    G
    \longrightarrow
    \GL(V)
\end{equation}
of $G$ by choosing a vector $v \in V$ and a linear functional $\varphi \in V'$ as
\begin{equation}
    \label{eq:RepresentativeFunction}
    \pi_{v,\varphi}
    \colon
    G
    \longrightarrow
    \C, \quad
    \pi_{v,\varphi}(g)
    \coloneqq
    \varphi
    \bigl(
        \pi(g)v
    \bigr).
\end{equation}
In particular, it is not clear whether the algebra of entire functions separates points
in general. As we shall see in Example~\ref{ex:AutomaticPeriodicity}, this is indeed
not always the case. It is the purpose of this article to describe~$\Entire(G)$ as
pullbacks of holomorphic functions on Hochschild's universal complexification
$G_\C$ of $G$, which he conceived within \cite{hochschild:1966a}. This
complexification consists of a pair~$(G_\C, \eta)$ of a complex Lie group $G_\C$ and
a Lie group morphism $\eta \colon G \longrightarrow G_\C$ fulfilling the following
universal property:

{
    \setlength{\parindent}{15pt}
    For every Lie group morphism $\Phi \colon G \longrightarrow H$ into a complex
    Lie group $H$, there exists a unique holomorphic Lie group morphism $\Phi_\C
    \colon G_\C \longrightarrow H$ such that $\Phi_\C \circ \eta = \Phi$.
}

We suppress the subscript and write $\eta$ instead when only one complexification
is involved. Moreover, the letter $\eta$ always encodes the morphism coming from a
universal complexification. In terms of a commutative diagram, the universal
complexification looks like this:
\begin{equation}
    \label{eq:UniversalComplexification}
    \begin{tikzcd}[column sep = huge, row sep = huge]
        G
        \arrow[d, "\eta"]
        \arrow[r, "\Phi"]
        &H \\
        G_\C
        \ar[swap,ru, dashed, "\exists_1 \; \Phi_\C"]
    \end{tikzcd}
\end{equation}
If $G_\C$ exists, then it is unique up to unique biholomorphic group morphism by
the usual arguments. Moreover, one may regard $\argument_\C$ as a functor from
the category of real Lie groups into the category of complex Lie groups. As such, the
diagram \eqref{eq:UniversalComplexification} consists of two different kinds of nodes
and arrows.

For our purposes, the crucial point is that pullback along~$\eta$ allows us to
compare functions on $G_\C$ with functions on $G$. Taking another look at the
definition of representative functions in \eqref{eq:RepresentativeFunction}, they
induce holomorphic mappings $\Pi_{v,\varphi}$ on~$G_\C$ satisfying
\begin{equation}
    \Pi_{v,\varphi}
    \circ
    \eta
    =
    \pi_{v,\varphi}
\end{equation}
by means of the universal property after complexifying the representation space $V$
if necessary. Theorem~\ref{thm:Extension} states that this extendability a general
phenomenon for entire functions. At the start of Section~\ref{sec:Extension}, we are
going to study the geometric fine structure of the universal complexification and
sketch a concrete construction.

Before delving into the precise definition of the algebra $\Entire(G)$ of entire
functions, it is instructive to review the classical Lie-theoretic incarnations of the
Taylor formula, as it can be found in the beginning of
\cite[Sec.~2.1.4]{helgason:2001a} for Lie groups and \cite[(1.48)]{forstneric:2011a} or
\cite[Prop.~4.5]{heins.roth.waldmann:2023a} for arbitrary analytic manifolds. In the
sequel, we write~$\Comega(G)$ for the algebra of real-analytic functions on a Lie
group~$G$. Here we use the fact that the exponential atlas endows $G$ with the
structure of an analytic manifold, see for instance
\cite[(1.6.3)~Theorem]{duistermaat.kolk:2000a}.

Recall that the Lie algebra $\liealg{g}$ acts on the space of complex-valued smooth
functions $\Cinfty(G)$ defined on $G$ by means of Lie derivatives in
direction of left-invariant vector fields. Identifying $\liealg{g} = T_\E G$ as the
tangent space at the group unit, we write~$X_\xi$ for the left-invariant vector field
with $X_\xi(\E) = \xi$ and denote its action by~$\Lie_{X_\xi}$. By the universal
property of the universal enveloping algebra $\Universal^\bullet(\liealg{g})$ of
$\liealg{g}$, this induces an algebra morphism
\begin{equation}
    \label{eq:LieDerivative}
    \Lie
    \colon
    \Universal^\bullet(\liealg{g})
    \longrightarrow
    \Diffop^\bullet
    \bigl(
        \Cinfty(G)
    \bigr)
\end{equation}
mapping into the algebra of left invariant differential operators acting on
$\Cinfty(G)$. In the sequel, we suppress the product of
$\Universal^\bullet(\liealg{g})$ within our notation.
\begin{proposition}[One variable Lie-Taylor formula
{\cite[Prop.~3.2.1]{heins:2024a}}]
    \label{prop:LieTaylorOneVariable}%
    Let $G$ be a Lie group and fix a real-analytic $\phi \in \Comega(G)$, an expansion
    point $g \in G$ and a direction $\xi \in \liealg{g}$.
    \begin{propositionlist}
        \item There exists a radius $R > 0$ such that the Lie-Taylor formula
        \begin{equation}
            \label{eq:LieTaylorOneVariable}
            \phi
            \bigl(
            g \cdot \exp(t\xi)
            \bigr)
            =
            \bigl(
            \exp(t \Lie_{X_\xi})
            \phi
            \bigr)(g)
            =
            \sum_{k=0}^{\infty}
            \frac{t^k}{k!}
            \bigl(
                \Lie_{X_\xi}^k
                \phi
            \bigr)
            (g)
        \end{equation}
        holds whenever $t \in \R$ fulfils $\abs{t} < R$.
        \item The series \eqref{eq:LieTaylorOneVariable} is convergent in the Fréchet
        space $\Holomorphic(\Ball_R(0))$ of holomorphic functions on the open disk
        $\Ball_R(0) \subseteq \C$ with respect to the complex variable~$t$.
    \end{propositionlist}
\end{proposition}

One should think of $g \in G$ as the expansion point of the power series, $\xi \in
\liealg{g}$ the direction and $t \in \R$ as the magnitude of the perturbation or,
alternatively, a time parameter. As some of the methods used for the proof are going
to be useful later, we provide our own.
\begin{proof}[of Proposition~\ref{prop:LieTaylorOneVariable}]
    Let $U \subseteq \liealg{g}$ be an open neighbourhood of $0$ such that
    \begin{equation}
        \exp
        \colon
        U \longrightarrow \exp(U)
    \end{equation}
    is a diffeomorphism. The idea is to consider the left-hand side of
    \eqref{eq:LieTaylorOneVariable} for $t\xi \in U$ as our function $\phi$
    precomposed with the exponential chart $\ell_g \circ \exp$ around $g$, where
    $\ell_g \colon G \longrightarrow G$ denotes the left multiplication with $g \in G$.
    By analyticity of $\phi$ and by shrinking $U$ if necessary, we may thus assume
    that $\phi \circ \ell_g \circ \exp \colon U \longrightarrow \C$ is globally given by its
    Taylor series around zero. Moreover, we may achieve $\xi = \basis{e}_1$ by
    choosing an appropriate basis~$(\basis{e}_1, \ldots, \basis{e}_d)$ of $\liealg{g}$. As
    the origin is an inner point of $U$, we find some $R > 0$ such that the open
    $\supnorm{\argument}$-ball $\Ball_{R}(0)$ is contained within $U$. Here,
    $\supnorm{\argument}$ denotes the supremum-norm defined by
    \begin{equation}
        \supnorm{\xi}
        \coloneqq
        \max_{n=1,\ldots,d}
        \abs[\big]
        {
            \basis{e}^n(\xi)
        },
    \end{equation}
    where $(\basis{e}^1, \ldots, \basis{e}^d)$ is the dual basis corresponding to
    $(\basis{e}_1, \ldots, \basis{e}_d)$. Geometrically, these balls constitute polydisks of
    polyradius $(R,\ldots,R)$. For $\norm{t\xi} = \abs{t} < R$, the classical Taylor
    formula around the origin then takes the simple form
    \begin{equation}
        \label{eq:LieTaylorProof}
        \phi
        \bigl(
        g \exp(t \xi)
        \bigr)
        =
        \sum_{k=0}^{\infty}
        \frac{t^k}{k!}
        \frac{\D^k}{\D s^k}
        \phi
        \bigl(
        g \exp(s \xi)
        \bigr)
        \at[\Big]{s=0}.
        \tag{$\ast$}
    \end{equation}
    By left-invariance, we have
    \begin{align}
        \Bigl(
        \Lie(\xi)
        \phi
        \Bigr)
        \bigl(
        g \exp(s\xi)
        \bigr)
        &=
        \bigl(
        \ell_{g \exp(s\xi)}^*
        \circ
        \Lie(\xi)
        \bigr)
        \phi
        \at[\Big]
        {\E} \\
        &=
        \bigl(
        \Lie(\xi)
        \circ
        \ell_{g \exp(s\xi)}^*
        \bigr)
        \phi
        \at[\Big]
        {\E} \\
        &=
        \frac{\D}{\D t}
        \phi
        \bigl(
        g \exp(s\xi) \exp(t\xi)
        \bigr)
        \at[\Big]{t=0} \\
        &=
        \frac{\D}{\D t}
        \phi
        \bigl(
        g \exp( (s + t) \xi)
        \bigr)
        \at[\Big]{t=0} \\
        &=
        \frac{\D}{\D s}
        \phi
        \bigl(
        g \exp(s\xi)
        \bigr)
    \end{align}
    and thus
    \begin{equation}
        \label{eq:LieDerivativeHigherOrder}
        \Lie_{X_\xi}^k
        \phi
        \at[\Big]
        {g \exp(s\xi)}
        =
        \frac{\D^k}{\D s^k}
        \phi
        \bigl(
        g \exp(s\xi)
        \bigr)
        \qquad
        \textrm{for all $k \in \N_0$.}
    \end{equation}
    Plugging this into \eqref{eq:LieTaylorProof} and once more invoking
    left-invariance, we obtain
    \begin{equation}
        \phi
        \bigl(
        g \exp(t \xi)
        \bigr)
        =
        \sum_{k=0}^{\infty}
        \frac{t^k}{k!}
        \Lie_{X_\xi}^k
        \phi
        \bigl(
        g \exp(s\xi)
        \bigr)
        \at[\Big]{s=0}
        =
        \sum_{k=0}^{\infty}
        \frac{t^k}{k!}
        \Lie_{X_\xi}^k
        \phi
        (g)
    \end{equation}
    whenever $\abs{t} < R$, which is \eqref{eq:LieTaylorOneVariable}. The remaining
    statement follows from the fact that the partial sums of
    \eqref{eq:LieTaylorOneVariable} correspond exactly to the partial sums of the
    classical Taylor series of $\phi \circ \ell_g \circ \exp$ and the latter is even
    convergent in the Fréchet space $\Holomorphic(\Ball_R(0))$ of holomorphic
    functions on the open disk with radius $R$.
\end{proof}

\begin{remark}[Uniform choices, {\cite[Cor.~4.9]{heins.roth.waldmann:2023a}}]
    \label{rem:UniformChoices}%
    We return to the setting of Proposition~\ref{prop:LieTaylorOneVariable}. By what
    we have shown, the function $R_g \colon \liealg{g} \longrightarrow (0,\infty) \cup
    \{\infty\}$,
    \begin{equation}
        R_g(\xi)
        \coloneqq
        \sup
        \bigl\{
        R > 0
        \colon
        \eqref{eq:LieTaylorOneVariable}
        \; \textrm{holds whenever} \;
        \abs{t} < R
        \bigr\}
    \end{equation}
    is well defined for every $g \in G$. Taking another look at our construction for
    fixed $\xi \in \liealg{g}$, we see that the zero neighbourhood $U$ was
    independent of $\xi$. Consequently,
    \begin{equation}
        R_g
        \coloneqq
        \min_{\supnorm{\chi} \le 1}
        R_g(\chi)
        \ge
        R_g(\xi)
        >
        0,
    \end{equation}
    where we take the supnorm with respect to some linear algebraic basis of
    $\liealg{g}$ with $\xi = \basis{e}_1$. Indeed, given $\chi \in \liealg{g}$
    with $\supnorm{\chi} \le 1$ that is not a
    multiple of $\xi$, we may change basis such that $\basis{e}_2 = \chi$. Having a
    uniform choice $R_g > 0$ reflects the fact that power series converge within open
    polydisks. We are going to incorporate this observation in our formulation of the
    multivariable Lie-Taylor formula. Note that $R_g$ depends on the choice of our
    basis and thus has no intrinsic geometric meaning. However,~$R_g > 0$ with
    respect to some basis does imply that~$R_g > 0$ for \emph{any} basis. This, in
    turn, is equivalent to $\phi$ being real-analytic at the point $g$.
\end{remark}

To simplify the bookkeeping, we fix some more notation, which was first established
within \cite[Sec.~4]{heins.roth.waldmann:2023a}. Recall that we have once and for
all fixed a basis $(\basis{e}_1, \ldots, \basis{e}_d)$ of $\liealg{g}$. We write
\begin{equation}
    \label{eq:LieDerivativeOfMulti}
    \Lie(\alpha)
    \coloneqq
    \Lie_{X_{\alpha_1}}
    \circ
    \cdots
    \circ
    \Lie_{X_{\alpha_n}}
\end{equation}
as well as
\begin{equation}
    \underline{z}^\alpha
    \coloneqq
    z^{\alpha_1}
    \cdots
    z^{\alpha_n}
    \quad \textrm{for} \quad
    \underline{z}
    =
    (z^1, \ldots, z^n)
    \in
    \C^n
\end{equation}
for any $n$-tuple $\alpha \in \{1,\ldots,d\}^n \eqqcolon \N_d^n$. This constitutes a
non-commutative version of the usual multi-index notation, which takes into
account that Lie derivatives into different directions typically do not commute. Given
a smooth function $\phi \in \Cinfty(G)$ and an expansion point $g \in G$, we then
call the formal power series
\begin{equation}
    \label{eq:LieTaylorFormal}
    \Taylor_\phi(g,\underline{z})
    \coloneqq
    \sum_{n=0}^{\infty}
    \frac{1}{n!}
    \sum_{\alpha \in \N_d^n}
    \bigl(
        \Lie(\alpha)\phi
    \bigr)(g)
    \cdot
    \underline{z}^\alpha
    \in
    \C\formal{\underline{z}}
\end{equation}
the Lie-Taylor series of $\phi$ at the point $g \in G$. This is warranted by the
following theorem. Note that we employ Einstein's summation convention.
\begin{theorem}[Lie-Taylor formula,
    {\cite[Cor.~4.6]{heins.roth.waldmann:2023a}}]
    \label{thm:LieTaylor}%
    Let $G$ be a Lie group, $g \in G$ and $\phi \in \Comega(G)$. Fix a basis
    $(\basis{e}_1, \ldots, \basis{e}_d)$ of $\liealg{g}$. Then there is a radius $R_g > 0$
    such that
    \begin{equation}
        \label{eq:LieTaylor}
        \phi
        \bigl(
            g
            \exp(x^k \basis{e}_k)
        \bigr)
        =
        \Taylor_\phi(\underline{x},g)
    \end{equation}
    for all $\underline{x} = (x^1, \ldots, x^d) \in \R^d$ with $\supnorm{\underline{x}}
    < R_g$.
\end{theorem}
\begin{proof}
    We define $R_g$ as discussed in Remark~\ref{rem:UniformChoices} corresponding
    to the basis $(\basis{e}_1, \ldots, \basis{e}_d)$. Consequently, we may apply
    \eqref{eq:LieTaylorOneVariable} to the direction $\xi \coloneqq x^k \basis{e}_k$
    for all $\underline{x} = (x^1, \ldots, x^d) \in \R^d$ with
    $\supnorm{\underline{x}} < R_g$. Establishing the identity \eqref{eq:LieTaylor} is
    now purely a matter of multiplying out within the universal enveloping algebra
    $\Universal^\bullet(\liealg{g})$. Indeed, we have
    \begin{equation}
        \xi^n
        =
        (x^k \basis{e}_k)^n
        =
        \sum_{\alpha \in \N_d^n}
        \basis{e}_{\alpha_n}
        \cdots
        \basis{e}_{\alpha_1}
        \cdot
        \underline{x}^\alpha
    \end{equation}
    and thus
    \begin{equation}
        \Lie(\xi)^n
        =
        \sum_{\alpha \in \N_d^n}
        \Lie(\alpha)
        \cdot
        \underline{x}^\alpha.
    \end{equation}
    Plugging this into \eqref{eq:LieTaylorOneVariable} proves \eqref{eq:LieTaylor} in
    view of \eqref{eq:LieTaylorFormal}.
\end{proof}

Pulling the absolute values all the way inside of \eqref{eq:LieTaylor}, we now define
the \emph{Lie-Taylor majorant} of $\phi \in \Cinfty(G)$ at $g \in G$ as the formal
power series
\begin{equation}
    \label{eq:LieTaylorMajorant}
    \Majorant_\phi(g,\underline{z})
    \coloneqq
    \sum_{n=0}^{\infty}
    c_{n}(\phi,g)
    \cdot
    \underline{z}^n
    \in
    \C\formal{\underline{z}},
\end{equation}
where
\begin{equation}
    \label{eq:LieTaylorMajorantCoefficients}
    c_n(\phi,g)
    \coloneqq
    \frac{1}{n!}
    \sum_{\alpha \in \N_d^n}
    \abs[\Big]
    {
        \bigl(
        \Lie(\alpha)\phi
        \bigr)(g)
    }
    \qquad
    \textrm{for all }
    n \in \N_0.
\end{equation}
As its name suggests, finiteness of the Lie-Taylor majorant
$\Majorant_\phi(g,\underline{z})$ implies the convergence of the Lie-Taylor series
$\Taylor_\phi(g,\underline{z})$ from \eqref{eq:LieTaylorFormal}. Note that we have
suppressed the choice of basis in our notation, as $\Majorant_\phi(g,\argument) \in
\Holomorphic(\C)$ for one basis implies that $\Majorant_\phi(g,\argument) \in
\Holomorphic(\C)$ for any basis. Similarly, it is the content of
Lemma~\ref{lem:MajorantTranslation} that $\Majorant_\phi(g,\argument) \in
\Holomorphic(\C)$ implies the same statement for all other points in the connected
component of $g$. This leads us to the notion of \emph{entire functions} on $G$.
That is to say, we demand that the Lie-Taylor majorant~$\Majorant_\phi$ at the
group unit defines an entire function of one complex variable.
\begin{definition}[Entire functions, {\cite[Def.~4.10]{heins.roth.waldmann:2023a}}]
    \label{def:EntireFunction}%
    Let $G$ be a connected Lie group.
    \begin{definitionlist}
        \item A real-analytic function $\phi \in \Comega(G)$ is called entire
        if its Lie-Taylor majorant
        \begin{equation}
            \Majorant_\phi(\E,\argument)
            \colon
            \C
            \longrightarrow
            \C
        \end{equation}
        at the group unit is holomorphic.
        \item We denote the set of all entire functions on $G$ by
        \begin{equation}
            \Entire(G)
            \coloneqq
            \bigl\{
                \phi \in \Comega(G)
                \colon
                \Majorant_{\phi}
                \in
                \Holomorphic(\C)
            \bigr\}.
        \end{equation}
        \item For $r \ge 0$, we define a family of seminorms on $\Entire(G)$ by
        \begin{equation}
            \label{eq:EntireSeminorms}
            \seminorm{q}_r(\phi)
            \coloneqq
            \Majorant_{\phi}(r)
        \end{equation}
        and endow $\Entire_0(G)$ with the corresponding locally convex topology.
    \end{definitionlist}
\end{definition}

The advantage of working with $\Majorant_\phi$ instead of with the seminorms
\eqref{eq:EntireSeminorms} is that $\Majorant_\phi$ is a holomorphic function of
one complex variable. This simple shift of perspective allows for the utilization of
many useful complex analytic techniques, such as the Cauchy estimates and
Montel's Theorem. A comprehensive discussion of $\Entire(G)$ based on this flexible
point of view can be found in \cite[Sec.~4.2 \&
Sec.~4.3]{heins.roth.waldmann:2023a}. We only note the following property, which is
crucial for the statement of our main results and also enters into their proofs.
\begin{proposition}[Functoriality of $\Entire$
{\cite[Prop.~4.21]{heins.roth.waldmann:2023a}}]
\label{prop:EntireFunctoriality}%
    Let $\Phi \colon G \longrightarrow H$ be a Lie group morphism between
    connected Lie groups $G$ and $H$. Then the pullback with $\Phi$ constitutes a
    morphism of Fréchet algebras
    \begin{equation}
        \label{eq:EntireFunctoriality}
        \Phi^*
        \colon
        \Entire(H)
        \longrightarrow
        \Entire(G).
    \end{equation}
\end{proposition}

As already noted, the purpose of this paper is to establish a one-to-one
correspondence between entire functions on $G$ and holomorphic ones on $G_\C$
for connected Lie groups. We split this into the problems of extension and restriction.
The latter simply poses the question, whether $\Holomorphic(G) \subseteq
\Entire(G)$ for a complex Lie group $G$. By the functoriality from
Proposition~\ref{prop:EntireFunctoriality}, a positive answer to this question implies
\begin{equation}
    \eta^*
    \bigl(
        \Holomorphic(G_\C)
    \bigr)
    \subseteq
    \Entire(G)
\end{equation}
for any real Lie group $G$. We provide the following positive resolution to the
restriction problem. Recall that an embedding $\iota \colon V \longrightarrow W$ of
locally convex spaces is a linear injection, which constitutes a homeomorphism onto
its image.
\begin{theorem}[Restriction]
    \label{thm:Restriction}
    Let $G$ be a complex connected Lie group. Then
    \begin{equation}
        \Holomorphic(G)
        \subseteq
        \Entire(G)
    \end{equation}
    induces an embedding of Fréchet algebras, where we endow $\Holomorphic(G)$
    with the topology of locally uniform convergence and products are taken
    pointwisely.
\end{theorem}

The proof is based on a Lie-theoretic version of the Cauchy estimates. We derive two
versions, one adapted to matrix Lie groups and operator norms, the other based on
left-invariant Riemannian geometry. The resulting estimates are interesting in their
own rights and facilitate an application from strict deformation quantization within
Theorem~\ref{thm:StarProductHolomorphicContinuity}.

Conversely, we consider extensions of entire functions $\phi$ on $G$ to
holomorphic mappings $\Phi$ on $G_\C$. Immediately, we run into the principal
problem that the universal complexification morphism $\eta \colon G
\longrightarrow G_\C$ needs not be injective, see
Example~\ref{ex:AutomaticPeriodicity} for a concrete and rather natural instance of
this phenomenon and its consequences in the context of the special linear group.
That being said, the pullback condition
\begin{equation}
    \label{eq:PullbackCondition}
    \eta^*
    \Phi
    \coloneqq
    \Phi
    \circ
    \eta
    =
    \phi
\end{equation}
remains meaningful, and this is what we mean by extension in the sequel.
The precise statement is the following.
\begin{theorem}[Extension]
    \label{thm:Extension}
    Let $G$ be a connected Lie group and let $\phi \in \Entire(G)$. Then there exists a
    unique holomorphic function
    \begin{equation}
        \Phi
        \in
        \Holomorphic(G_\C)
        \qquad
        \textrm{such that}
        \qquad
        \eta^* \Phi
        =
        \phi.
    \end{equation}
    It fulfils
    \begin{equation}
        \label{eq:ExtensionLieTaylor}
        \Lie(\xi)
        \Phi
        \at[\Big]{\eta(g)}
        =
        \Lie(\xi)
        \phi
        \at[\Big]{g}
        \qquad
        \textrm{for all }
        g \in G,
        \xi \in \Universal^\bullet
        (\liealg{g}_\C).
    \end{equation}
\end{theorem}

In \eqref{eq:ExtensionLieTaylor}, we use that the vector space complexification
$\liealg{g}_\C$ of the Lie algebra always acts on $\Cinfty(G)$, even when $\eta$ is
not locally injective. The proof of Theorem~\ref{thm:Extension} is more involved than
the one of Theorem~\ref{thm:Restriction}, as it uses some geometric ingredients.
That being said, the rough strategy is the following. Starting with an entire
function~$\phi \in \Entire(G)$, we consider the pullback
\begin{equation}
    \psi
    \coloneqq
    p^* \phi
    \in
    \Entire(\Covering)
\end{equation}
to the universal covering group $p \colon \Covering \longrightarrow G$ of $G$. A
comprehensive discussion of coverings in the context of Lie theory can be found in
\cite[Appendix~A \& Sec.~9.4]{hilgert.neeb:2012a}. Using the local injectivity of the
universal complexification morphism $\tilde{\eta} \colon \Covering \longrightarrow
\Covering_\C$ in conjunction with the Lie-Taylor formula \eqref{eq:LieTaylor}, we
then define a \emph{holomorphic shadow} $\psi_0$ of~$\psi$ on a neighbourhood
of the group unit in $\Covering_\C$. This reduces the pullback
condition~\eqref{eq:PullbackCondition} to an honest problem about analytic
continuation. After porting some classical complex analysis into a left invariant guise,
we then construct a holomorphic extension of $\psi_0$ to a globally defined
holomorphic function on $\Covering$ by means of the Monodromy Theorem. As by
construction, $\psi$ is $\pi_1(G)$-periodic, we finally use the identity principle to
prove the required periodicity to ensure the descent of $\Psi$ to a holomorpic
function on $G_\C$. This mapping then provides the desired holomorphic extension
$\Phi$ of $\phi$.

Combining our results, we may establish that extension and restriction are inverse
processes and provide an isomorphism between the Fréchet algebras $\Entire(G)$
and $\Holomorphic(G_\field{C})$.
\begin{theorem}[Entire vs. Holomorphic]
    \label{thm:EntireVsHolomorphic}%
    Let $G$ be a connected Lie group. Then the pullback
    \begin{equation}
        \label{eq:EntireVsHolomorphicIso}
        \eta^*
        \colon
        \Holomorphic(G_\C)
        \longrightarrow
        \Entire(G), \quad
        \eta^*
        \Phi
        \coloneqq
        \Phi \circ \eta
    \end{equation}
    with the universal complexification morphism $\eta$ is well defined and an
    isomorphism of Fréchet algebras, where the products are defined pointwisely.
\end{theorem}
This provides another strategy to obtain the pleasant functional analytic properties
of~$\Entire(G)$ discussed in \cite[Theorem~4.18]{heins.roth.waldmann:2023a}.
Assuming the validity of Theorem~\ref{thm:Restriction} and
Theorem~\ref{thm:Extension} for the moment, it is straightforward to prove
Theorem~\ref{thm:EntireVsHolomorphic}.
\begin{proof}
    Given $\Phi \in
    \Holomorphic(G_\C)$, Theorem~\ref{thm:Restriction} yields $\Phi \in
    \Entire(G_\C)$ and thus $\eta^* \Phi \in \Entire(G)$ by the functoriality
    of~$\Entire$ from Proposition~\ref{prop:EntireFunctoriality}.
    Consequently, Theorem~\ref{thm:Extension} asserts the existence of a unique
    holomorphic mapping $\Psi \in \Holomorphic(G_\C)$ such that
    \begin{equation}
        \eta^*
        \Psi
        =
        \eta^* \Phi.
    \end{equation}
    But, trivially, $\Phi$ also has these properties, which implies $\Psi = \Phi$ by
    uniqueness. Conversely, let $\phi \in \Entire(G)$ be given. By
    Theorem~\ref{thm:Extension}, there exists a unique holomorphic function
    \begin{equation}
        \Phi \in \Holomorphic(G_\C)
        \qquad \textrm{such that} \qquad
        \eta^* \Phi = \phi.
    \end{equation}
    That is to say, the restriction of $\Phi$ is just $\phi$ again. Consequently,
    \eqref{eq:EntireVsHolomorphicIso} is well defined and bijective.
    Combining the continuity of the inclusion $\Holomorphic(G_\C) \subseteq
    \Entire(G_\C)$ from Theorem~\ref{thm:Restriction} with the continuity of
    \begin{equation}
        \eta^*
        \colon
        \Entire(G_\C)
        \longrightarrow
        \Entire(G)
    \end{equation}
    from Proposition~\ref{prop:EntireFunctoriality} now establishes the continuity
    of \eqref{eq:EntireVsHolomorphicIso}. Finally, as we are dealing with Fréchet
    spaces, the Open Mapping Theorem \cite[Theorem~4.35]{osborne:2014a} implies
    that the continuity of the inverse is automatic. Alternatively, an explicit estimate
    may be derived by combining~\eqref{eq:ExtensionLieTaylor}, which asserts that
    extension preserves Lie-Taylor data, with the fact that the subspace topology of
    $\Holomorphic(G_\C)$ induced by $\Entire(G_\C)$ is the
    $\Holomorphic$-topology.
\end{proof}

We have thus demystified the structure of $\Entire(G)$.
\begin{remark}[$R$-Entire Functions]
    \label{rem:PositiveR}%
    Notably, \cite{heins.roth.waldmann:2023a} introduced not only $\Entire(G)$, but
    rather a spectrum of algebras $\Entire_R(G)$, where $R \ge 0$ is a parameter. This
    is achieved by altering the Lie-Taylor majorant \eqref{eq:LieTaylorMajorant} by
    means of a weight $n!^R$, which leads to
\begin{equation}
    \Majorant_\phi^{(R)}(g,\underline{z})
    \coloneqq
    \sum_{n=0}^\infty
    n!^R
    \cdot
    c_n(\phi,g)
    \cdot
    \underline{z}^n
    \in
    \C\formal{\underline{z}}.
\end{equation}
What we call $\Entire(G)$ then corresponds to $R=0$. As $\Entire_R(G) \subseteq
\Entire(G)$ for all $R \ge 0$, we know that every $R$-entire function comes from a
holomorphic mapping defined on $G_\C$ by Theorem~\ref{thm:Extension}. Notably,
the space $\Entire_R(\R^d)$ corresponds to the space of entire
functions~$\Holomorphic_{1/R}(\C^d)$ of finite order $\le 1/R$ and of minimal type
by \cite[Prop.~4.15~\textit{i.)} \& \textit{ii.)}]{heins.roth.waldmann:2023a}. We refer to
the textbook \cite{lelong.gruman:1986a} for a systematic treatment of the classical
theory. Using Theorem~\ref{thm:EntireVsHolomorphic}, we propose the definition
\begin{equation}
    \Holomorphic_{R}(G_\C)
    \coloneqq
    \bigl(
        \eta^*
    \bigr)^{-1}
    (
        \Entire_{1/R}(G)
    )
\end{equation}
    for the entire functions of finite order at most $R$ and of minimal type on $G_\C$.
    It would be interesting to describe these spaces intrinsically on $G_\C$ by means of
    suprema with respect to a suitable weight function. Moreover, it is not clear
    whether
    \begin{equation}
        \Holomorphic_R(G_\C)
        \neq \{0\}
        \qquad
        \textrm{for }
        R \ge 1,
    \end{equation}
     as the representative functions \eqref{eq:RepresentativeFunction} are contained
     within $\Entire_R(G)$ only for $R < 1$ in general by \cite[Thm.~4.23 \&
     Ex.~4.24]{heins.roth.waldmann:2023a}.
\end{remark}

Throughout the text, we are going to borrow some well-known methods from
Riemannian geometry, which we collect here for the sake of precision and the
convenience of the reader.
\begin{proposition}[Left-invariant Riemannian Geometry]
    \label{prop:RiemannianLeftInvariant}%
    Let $G$ be a connected Lie group and $\mathbb{g}_\E \in \liealg{g} \vee \liealg{g}$
    a non-degenerate symmetric two-form.
    \begin{propositionlist}
        \item \label{item:LeftTranslation}%
        Left translation of $\mathbb{g}_\E$ defines a left-invariant Riemannian
        metric $\mathbb{g}$ on $G$, which is explicitly given by
        \begin{equation}
            \label{eq:RiemannianLeftTranslation}
            \mathbb{g}
            \at[\Big]{g}
            (v_g,w_g)
            \coloneqq
            \ell^*_{g^{-1}}
            \mathbb{g}_\E
            (v_g, w_p)
            =
            \mathbb{g}_\E
            \bigl(
                T_\E \ell_g^{-1} v_g,
                T_\E \ell_g^{-1} w_g
            \bigr)
        \end{equation}
        for all $g \in G$ and $v_g,w_g \in T_g G$, where
        \begin{equation}
            T_\E \ell_g
            \colon
            \liealg{g}
            =
            T_\E G
            \longrightarrow
            T_g G
        \end{equation}
        denotes the tangent map of $\ell_g$.

        \item \label{item:RiemannianLengthInvariance}%
        The corresponding Riemannian length $\varrho$ on $G$ is left-invariant, i.e.
        fulfils
        \begin{equation}
            \label{eq:RiemannianLengthInvariance}
            \varrho
            \bigl(
                hg_1,
                hg_2
            \bigr)
            =
            \varrho
            \bigl(
                \ell_h g_1,
                \ell_h g_2
            \bigr)
            =
            \varrho(g_1,g_2)
            \qquad
            \textrm{for all }
            g_1,g_2,h \in G.
        \end{equation}

        \item \label{item:RiemannianLengthTopology}%
        The metric topology of $(G,\varrho)$ coincides with the original topology of
        $G$.

        \item \label{item:RiemannianLengthMultiplicativity}%
        We have
        \begin{equation}
            \label{eq:RiemannianLengthMultiplicativity}
            \varrho
            \bigl(
                \E, g_1 \cdots g_n
            \bigr)
            \le
            \varrho(\E, g_1)
            +
            \cdots
            +
            \varrho(\E, g_n)
            \qquad
            \textrm{for all }
            g_1, \ldots, g_n \in G.
        \end{equation}

        \item \label{item:RiemannianLengthLipschitz}%
        Let $\norm{\argument}$ be the norm induced by $\mathbb{g}$. Then
        \begin{equation}
            \label{eq:RiemannianLengthLipschitz}
            \varrho
            \bigl(
                \E, \exp \xi
            \bigr)
            \le
            \norm{\xi}
            \qquad
            \textrm{for all }
            \xi \in \liealg{g}.
        \end{equation}
    \end{propositionlist}
\end{proposition}
\begin{proof}
    A discussion of left-translations of vector fields can be found in
    \cite[Sec.~9.1.2]{hilgert.neeb:2012a} and the invariance is readily dualized and
    generalized to arbitrary tensor fields. Taking another look at
    \eqref{eq:RiemannianLeftTranslation}, the symmetry of $\mathbb{g}$ is clear and
    non-degeneracy follows from the injectivity of $T_\E \ell_g^{-1}$. This establishes
    \ref{item:LeftTranslation}. For the second part, recall that the Riemannian length of
    a smooth curve $\gamma \colon [0,1] \longrightarrow G$ is given by
    \begin{equation}
        \label{eq:RiemannianLength}
        L(\gamma)
        =
        \int_0^1
        \mathbb{g}
        \at[\Big]{\gamma(t)}
        \bigl(
            \gamma'(t),
            \gamma'(t)
        \bigr)^{1/2}
        \D t.
    \end{equation}
    Let now $g \in G$ and consider the translated curve $\ell_g \gamma$ given by
    $\ell_g \gamma(t) \coloneqq g \cdot \gamma(t)$ for all~$t \in [0,1]$. Then
    \begin{equation}
        \bigl(
            \ell_g \gamma
        \bigr)'(t)
        =
        T_{\gamma(t)}
        \ell_g
        \gamma'(t)
        \qquad
        \textrm{for all }
        t \in [0,1],
    \end{equation}
    as
    \begin{equation}
        \bigl(
            \ell_g \gamma
        \bigr)'(t)
        \phi
        =
        \frac{\D}{\D s}
        \phi
        \bigl(
            g \gamma(t+s)
        \bigr)
        \at[\Big]{s=0}
        =
        \frac{\D}{\D s}
            \ell_g^* \phi
            \bigl(
                \gamma(t+s)
            \bigr)
        \at[\Big]{s=0}
        =
        \bigl(
            T_{\gamma(t)}
            \ell_g
            \gamma'(t)
        \bigr)
        \phi
    \end{equation}
    for all smooth germs $\phi \in \Cinfty_{g \gamma(t)}(G)$ at $g\gamma(t)$. The
    left-invariance of $\mathbb{g}$ thus yields
    \begin{align}
        L(\ell_g \gamma)
        &=
        \int_0^1
        \mathbb{g}
        \at[\Big]{\gamma(t)}
        \bigl(
            T_{\gamma(t)}
            \ell_g
            \gamma'(t),
            T_{\gamma(t)}
            \ell_g
            \gamma'(t)
        \bigr)^{1/2}
        \D t \\
        &=
        \int_0^1
        \ell_g^* \mathbb{g}
        \at[\Big]{\gamma(t)}
        \bigl(
            \gamma'(t),
            \gamma'(t)
        \bigr)^{1/2}
        \D t \\
        &=
        \int_0^1
        \mathbb{g}
        \at[\Big]{\gamma(t)}
        \bigl(
        \gamma'(t),
        \gamma'(t)
        \bigr)^{1/2}
        \D t \\
        &=
        L(\gamma).
    \end{align}
    Recall now that the Riemannian distance is given by
    \begin{equation}
        \label{eq:RiemannianLengthMetric}
        \varrho(g_1,g_2)
        =
        \inf_{\gamma \in \Gamma(g_1,g_2)}
        L(\gamma)
        \qquad
        \textrm{for all }
        g_1,g_2 \in G,
    \end{equation}
    where $\Gamma(g,h)$ denotes the set of all smooth curves connecting $g$ with
    $h$. As the passage from $\gamma$ to $\ell_g \gamma$ constitutes a bijection
    between $\Gamma(g_1,g_2)$ and $\Gamma(g g_1, g g_2)$, this
    establishes~\ref{item:RiemannianLengthInvariance}.
    Part~\ref{item:RiemannianLengthTopology} is true on any Riemannian manifold
    by \cite[Theorem~13.29]{lee:2012a}. We prove
    \ref{item:RiemannianLengthMultiplicativity} by means of a simple induction. For
    $n=1$, there is nothing to be shown, and assuming the statement for $n-1$
    factors yields
    \begin{equation}
        \varrho
        \bigl(
            \E, g_1 \cdots g_n
        \bigr)
        =
        \varrho
        \bigl(
            g_1^{-1},
            g_2 \cdots g_n
        \bigr)
        \le
        \varrho
        \bigl(
            g_1^{-1},
            \E
        \bigr)
        +
        \varrho
        \bigl(
            \E,
            g_2 \cdots g_n
        \bigr)
        \le
        \varrho(\E, g_1)
        +
        \cdots
        +
        \varrho(\E, g_n),
    \end{equation}
    where we have invoked \eqref{eq:RiemannianLengthInvariance} several times and
    used the triangle inequality. Finally, note that \ref{item:RiemannianLengthLipschitz}
    may be interpreted as the Lipschitz continuity of
    \begin{equation}
        \exp
        \colon
        (K, \norm{\argument})
        \longrightarrow
        (G,\varrho)
    \end{equation}
    evaluated at the pair~$(0,\xi)$. Fix $\xi \in K$ and consider the curve
    \begin{equation}
        \gamma
        \colon [0,1]
        \longrightarrow
        G, \quad
        \gamma(t)
        \coloneqq
        \exp(t\xi).
    \end{equation}
    As our curve $\gamma$ connects $\E$ with $\exp \xi$, we get
    \begin{align}
        \varrho
        \bigl(
            \E, \exp \xi
        \bigr)
        &\le
        L(\gamma) \\
        &=
        \int_0^1
        \mathbb{g}
        \at[\Big]
        {\exp(t\xi)}
        \bigl(
            T_\E \ell_{\exp(t \xi)} \xi,
            T_\E \ell_{\exp(t \xi)} \xi
        \bigr)^{1/2}
        \D t \\
        &=
        \int_0^1
        \bigl(
            \ell_{\exp(t \xi)}^* \mathbb{g}
        \bigr)
        \at[\Big]{\E}
        (\xi,\xi)^{1/2}
        \D t \\
        &=
        \int_0^1
        \mathbb{g}
        \at[\Big]{\E}
        (\xi,\xi)^{1/2}
        \D t \\
        &=
        \mathbb{g}_\E
        (\xi,\xi)^{1/2} \\
        &=
        \norm{\xi}
    \end{align}
    by virtue of the infimum in \eqref{eq:RiemannianLengthMetric},
    \eqref{eq:RiemannianLength} and the left-invariance of $\mathbb{g}$.
\end{proof}

If $G$ is compact, then one can average the pullbacks $r_g^* \mathbb{g}$ by the
right multiplications $r_g$ with $g \in G$ over the group to obtain a bi-invariant
metric. In this case, the Riemannian geodesics are precisely given by $t \mapsto g
\exp(t\xi)$ for $g \in G$ and $\xi \in \liealg{g}$ by \cite[Exercise~15.5.]{tu:2017a}. Our
proof shows that \eqref{eq:RiemannianLengthLipschitz} then becomes an equality.

\section{Extension: Function Theory on $G_\C$}
\label{sec:Extension}%

We collect the insights into the fine structure of $G_\C$, which we require for the
proof of Theorem~\ref{thm:Extension}. This knowledge then allows us to recast
many results from classical complex analysis into a left-invariant guise. More
comprehensive discussions including numerous examples can
be found in \cite[Sec.~15.1]{hilgert.neeb:2012a}.
\begin{proposition}[Geometry of $G_\C$
{\cite[Theorem~15.1.4]{hilgert.neeb:2012a}}] \; \\
    \label{prop:UniversalComplexification}%
    Let $G$ be a connected Lie group with universal covering group $p \colon
    \Covering \longrightarrow G$ and fundamental group $\pi_1(G) \coloneqq \ker p$.
    \begin{propositionlist}
        \item \label{item:ComplexificationCovering}%
        The universal complexification $(\Covering_\C, \tilde{\eta})$ of $\Covering$ may
        be realized as the unique connected and simply connected complex Lie group
        with complex Lie algebra $\liealg{g}_\C \coloneqq \liealg{g} \tensor_\R \C$ and
        the integration of the canonical inclusion of $\liealg{g}$ into $\liealg{g}_\C$. In
        particular,
        \begin{equation}
            \tilde{\eta}
            \colon
            \tilde{G}
            \longrightarrow
            \tilde{G}_\C
        \end{equation}
        is locally injective.

        \item \label{item:ComplexificationAsQuotient}%
        The universal complexification $(G_\C,\eta)$ of $G$ may be realized as the
        quotient
        \begin{equation}
            \label{eq:ComplexificationFromUniversalCovering}
            G_\C
            \coloneqq
            \widetilde{G}_\C
            \Big/
            \big<
            \widetilde{\eta}(\pi_1(G))
            \big>_{\C}
            \quad \textrm{and} \quad
            \eta
            \colon
            G
            \longrightarrow
            G_\C, \quad
            \eta
            \bigl(
            p(g)
            \bigr)
            \coloneqq
            \pi
            \bigl(
            \widetilde{\eta}(g)
            \bigr)
        \end{equation}
        with the quotient projection $\pi \colon \tilde{G}_\C \longrightarrow G_\C$ and
        where $\langle \argument \rangle_\C$ denotes the passage to the smallest
        closed complex subgroup of $G_\C$ containing the argument. In particular, the
        universal complexification $G_\C$ is connected itself and
        \begin{equation}
            \label{eq:ComplexificationDimension}
            \dim_\C(G_\C)
            \le
            \dim_\R(G).
        \end{equation}
    \end{propositionlist}
\end{proposition}
\begin{proof}
    The statements constitute the essence of the standard construction.
\end{proof}

Whenever we speak of $G_\C$ in the sequel, we refer to this specific realization. In
particular, this allows us to think of universal complexification as a functor. In terms
of a commutative diagram, our realization of $G_\C$ is given by
\begin{equation}
    \label{eq:ComplexificationFromUniversalCoveringDiagram}
    \begin{tikzcd}[column sep = huge]
        \widetilde{G}
        \arrow[d, "p", swap]
        \arrow[r, "\tilde{\eta}"]
        &\tilde{G}_\C
        \arrow[d, "\pi"] \\
        G
        \arrow[r, "\eta"]
        &G_\C
    \end{tikzcd}.
\end{equation}
The possibility of a dimension drop of $G_\C$ compared to $G$ in the sense of a
strict inequality within \eqref{eq:ComplexificationDimension} stem from passing
to the closure within~\eqref{eq:ComplexificationFromUniversalCovering}. A simple
example based on an infinite cyclic covering can be found in
\cite[Exercise~15.1.4]{hilgert.neeb:2012a}.

As noted before, the first step of our proof is to establish an honest problem about
analytic continuation. While $\eta$ itself may not be injective, the embedding of
$\eta(G)$ into $G_\C$ always is. Conveniently, $\eta(G)$ moreover carries the
structure of a Lie group itself, and we may prove that it complexifies to~$G_\C$
simply by checking the universal property.
\begin{lemma}
    \label{lem:ImageOfEta}%
    Let $G$ be a connected Lie group with universal complexification
    $(G_\field{C},\eta)$.
    \begin{lemmalist}
        \item \label{item:ImageOfEtaSubmanifold}%
        The image
        \begin{equation}
            \eta(G) \subseteq G_\C
        \end{equation}
        is a closed real Lie subgroup when endowed with the subspace topology.

        \item \label{item:ImageOfEtaComplexification}%
        The embedding $\iota \colon \eta(G) \hookrightarrow G_\C$ constitutes a
        universal complexification.

        \item \label{item:ImageOfEtaDimension}%
        Let $d \coloneqq \dim_\R(\eta(G))$. Then every basis
        \begin{equation}
            \bigl(
            \basis{e}_1, \ldots, \basis{e}_d
            \bigr)
            \subseteq
            \LieAlg\bigl(\eta(G)\bigr)
        \end{equation}
        induces a basis of $\LieAlg(G_\C)$ via $(\basis{e}_1, \ldots, \basis{e}_d, \I
        \basis{e}_1, \ldots, \I \basis{e}_d)$. In particular,
        \begin{equation}
            \label{eq:ImageOfEtaDimension}
            \dim_\C(G_\C)
            =
            \dim_\R
            \bigl(
                \eta(G)
            \bigr).
        \end{equation}

        \item \label{item:ComplexificationLocallyInjective}%
        If $\eta$ is locally injective, then
        \begin{equation}
            \label{eq:ComplexificationLocallyInjective}
            \LieAlg(G_\C)
            \cong
            \liealg{g}_\C.
        \end{equation}
    \end{lemmalist}
\end{lemma}
\begin{proof}
    Invoking \cite[Theorem~15.1.4, (iv)]{hilgert.neeb:2012a}, we find a
    antiholomorphic involutive group morphism
    \begin{equation}
        \sigma
        \colon
        G_\field{C}
        \longrightarrow
        G_\field{C}
        \qquad \textrm{such that} \qquad
        \eta(G)
        =
        \bigl\{
        g
        \in
        G_\field{C}
        \colon
        \sigma(g)
        =
        g
        \bigr\}_0,
    \end{equation}
    where the subscript denotes passing to the connected component of the group
    unit. By continuity of $\sigma$, its fixed point group is closed within $G_\field{C}$.
    As connected components of closed sets are themselves closed, this establishes
    the closedness of $\eta(G)$ within $G_\field{C}$. As images of group
    morphisms are subgroups, this in turn implies that $\eta(G)$ is a real Lie group
    itself by the Closed Subgroup Theorem \cite[Thm.~9.3.7]{hilgert.neeb:2012a}. Here
    we view the ambient group~$G_\field{C}$ as a real Lie group. In particular,
    $\eta(G)$ constitutes a submanifold of $G_\C$ with exponential charts as
    submanifold charts. In order to prove \ref{item:ImageOfEtaComplexification}, we
    check the universal property directly. To this end, let $\Phi \colon \eta(G)
    \longrightarrow H$ be a Lie group morphism into a complex Lie group $H$. Then
    \begin{equation}
        \Psi
        \coloneqq
        \Phi
        \circ
        \eta
        \colon
        G
        \longrightarrow
        H
    \end{equation}
    is a Lie group morphism from $G$ into a complex Lie group. Invoking the universal
    property of~$(G_\C, \iota \circ \eta)$, there thus exists a unique holomorphic
    group morphism
    \begin{equation}
        \Psi_\C \colon G_\C \longrightarrow H
        \qquad
        \textrm{with}
        \qquad
        \Psi_\C
        \circ
        \iota
        \circ
        \eta
        =
        \Psi
        =
        \Phi
        \circ
        \eta.
    \end{equation}
    By surjectivity of $\eta \colon G \longrightarrow \eta(G)$, this is equivalent to
    $\Psi_\C \circ \iota = \Phi$. That is to say, $\Psi_\C$ is the unique holomorphic
    mapping complexifying $\Phi$.

    Next, we prove \ref{item:ImageOfEtaDimension}, where we write $H \coloneqq
    \eta(G)$ to declutter the notation.
    Translating~\eqref{eq:ComplexificationFromUniversalCoveringDiagram} to $H$
    yields the commutative square
     \begin{equation}
         \label{eq:ComplexificationImageOfEta}
         \begin{tikzcd}[column sep = huge]
             \Covering[H]
             \arrow[d, "p", swap]
             \arrow[r, "\tilde{\iota}"]
             &\Covering[H]_\C
             \arrow[d, "\pi"] \\
             H
             \arrow[r, "\iota"]
             &H_\C \cong G_\C
         \end{tikzcd}
     \end{equation}
     in view of \ref{item:ImageOfEtaComplexification}. Thus, the composition~$\pi
     \circ \tilde{\iota}$ is locally injective by local injectivity of $p$ and injectivity of
     $\iota$. Hence, the tangent map
         \begin{equation}
                 T_\E
                 (
                    \pi
                    \circ
                    \tilde{\iota}
                 )
                 \colon
                 \LieAlg(\Covering[H])
                 \longrightarrow
                 \LieAlg(G_\C)
             \end{equation}
          is an injective Lie algebra morphism. But $\LieAlg(\Covering[H]) = \LieAlg(H)$,
          completing the proof of~\ref{item:ImageOfEtaDimension}. Finally,
          the assumed local injectivity on the one hand implies injectivity of the tangent
          map
          \begin{equation}
              T_\E
              \eta
              \colon
              \liealg{g}
              \longrightarrow
              \LieAlg(H).
          \end{equation}
          In particular, $\dim \liealg{g} \le \dim \LieAlg(H)$. On the other hand,
          combining \eqref{eq:ComplexificationDimension} with
          \eqref{eq:ImageOfEtaDimension} we know that
          \begin{equation}
              \dim_\R(G)
              \ge
              \dim_\C(G_\C)
              =
              \dim_\R(H).
          \end{equation}
          Thus, $\dim \liealg{g} = \dim \LieAlg(H)$ and thus the linear mapping~$T_\E
          \eta$ constitutes a bijection.
          Hence,~\ref{item:ComplexificationLocallyInjective}
          follows from \ref{item:ImageOfEtaDimension}.
\end{proof}

Using similar techniques and being a bit more careful, one can establish that
\begin{equation}
    \pi
    \colon
    \Covering_\C
    \longrightarrow
    G_\C
\end{equation}
is a universal covering projection whenever $\tilde{\eta} \colon \Covering
\longrightarrow \Covering_\C$ is injective. That is to say, in view of
\eqref{eq:ComplexificationFromUniversalCoveringDiagram}, the passage to the
universal covering group commutes with complexification. In particular,
\begin{equation}
    \pi_1(G_\C)
    \cong
    \pi_1(G)
\end{equation}
in this case. Dropping the assumption of injectivity, this equality may reduce to
the inclusion
\begin{equation}
    \pi_1(G_\C)
    \subseteq
    \pi_1(G).
\end{equation}
Either way, the upshot is that there are no new topological obstructions to analytic
continuation to $G_\C$ beyond what is already present within $G$. While this
provides heuristical motivation for Theorem~\ref{thm:Extension}, our proof only
requires the following left-invariant version of the identity principle.
\begin{corollary}[Identity Principle]
    \label{cor:IdentityPrinciple}%
    Let $G$ be a connected Lie group, $U \subseteq G_\C$ a domain with $U \cap
    \eta(G) \neq \emptyset$ and $\phi \in \Holomorphic(U,M)$ be a holomorphic map
    with values in a complex manifold $M$. If $\phi \at{U \cap \eta(G)}$ is constant,
    then so is $\phi$.
\end{corollary}
\begin{proof}
    Let $(\basis{e}_1, \ldots, \basis{e}_d)$ be a basis of $\LieAlg(\eta(G))$ with the
    induced basis
    \begin{equation}
        (
            \basis{e}_1,
            \ldots,
            \basis{e}_d,
            \I \basis{e}_1,
            \ldots,
            \I \basis{e}_d
         )
    \end{equation}
    of $\LieAlg(G_\C)$ from Lemma~\ref{lem:ImageOfEta},
    \ref{item:ImageOfEtaDimension}. By holomorphicity of $\phi$, we get the
    left-invariant Cauchy-Riemann equations
    \begin{equation}
        \label{eq:CauchyRiemann}
        \Lie(\I \basis{e}_n) \phi
        =
        \I \cdot \Lie(\basis{e}_n) \phi
        \qquad
        \textrm{for }
        n=1,\ldots,d.
    \end{equation}
    Indeed, as the Lie exponential is the flow of the left-invariant vector field, we have
    \begin{equation}
        \label{eq:LieDerivativeFlow}
        \Lie(\basis{e}_n)
        \phi
        \at[\Big]{g}
        =
        \frac{\D}{\D t}
        \phi
        \bigl(
            g \exp(t \basis{e}_n)
        \bigr)
        \at[\Big]{t=0}
    \end{equation}
    for all $g \in U$ and $n=1,\ldots,d$. This reveals \eqref{eq:CauchyRiemann} as
    the Cauchy-Riemann equations for the holomorphic mapping
    \begin{equation}
        \phi
        \circ
        \exp
        \colon
        \LieAlg(G_\C)
        =
        \LieAlg
        \bigl(
            \eta(G)
        \bigr)_\C
        \longrightarrow
        \C.
    \end{equation}
    Now, by assumption, $\Lie(e_n) \phi \at{U \cap \eta(G)} \equiv 0$ and
    \eqref{eq:CauchyRiemann} implies
    \begin{equation}
        \Lie(\I \basis{e}_n)\phi
        \at[\Big]{U \cap \eta(G)}
        \equiv 0
        \qquad
        \textrm{for }
        n=1,\ldots,d.
    \end{equation}
    Invoking the Lie-Taylor formula \eqref{eq:LieTaylor}, this means that $\phi$ is
    constant in an open neighbourhood of $U \cap \eta(G)$. By the usual identity
    principle, this completes the proof.
\end{proof}

Having established these geometric preliminaries, we may now reduce the extension
problem to an honest inquiry about analytic continuation. Here, we once again make
the additional assumption of local injectivity of $\eta$. By what we have shown in
Lemma~\ref{lem:ImageOfEta}, we may extend its tangent map to a $\C$-linear
isomorphism
\begin{equation}
    \label{eq:TangentMapExtension}
    (T_\E \eta)^{-1}
    \colon
    \LieAlg(G_\C)
    \longrightarrow
    \liealg{g}_\C.
\end{equation}
This facilitates the following construction of a local extension.
\begin{proposition}[Holomorphic Shadows]
    \label{lem:HolomorphicShadows}%
    Let $G$ be a connected Lie group such that
    \begin{equation}
        \eta
        \colon
        G
        \longrightarrow
        G_\C
    \end{equation}
    is locally injective and $\phi \in \Entire(G)$ an entire function. Then there exist an
    open zero neighbourhood $V \subseteq \LieAlg(G_\C)$ and a holomorphic
    mapping
    \begin{equation}
        \phi_0
        \colon
        U
        \coloneqq
        \exp_{G_\C} V
        \longrightarrow
        \C
    \end{equation}
    such that
    \begin{equation}
        \label{eq:HolomorphicShadowPullback}
        \eta^*
        \phi_0
        \at[\Big]
        {
            \eta^{-1}(U)
        }
        =
        \phi
        \at[\Big]{\eta^{-1}(U)}
    \end{equation}
    and
    \begin{equation}
        \label{eq:HolomorphicShadow}
        \phi_0
        \bigl(
            \exp_{G_\C} \xi
        \bigr)
        \coloneqq
        \Taylor_\phi
        \bigl(
            \E, \underline{(T_\E \eta)^{-1} \xi}
        \bigr)
        \qquad
        \textrm{for all }
        \xi \in V.
    \end{equation}
    In particular,
    \begin{equation}
        \label{eq:HolomorphicShadowCoefficients}%
        \Lie(\xi)\phi_0
        \at[\Big]{\eta(g)}
        =
        \Lie(\xi)\phi
        \at[\Big]{g}
        \qquad
        \textrm{for all }
        g \in \eta^{-1}(U),
        \xi \in \Universal^\bullet(\liealg{g}_\C).
    \end{equation}
\end{proposition}
\begin{proof}
    The idea is to define $\phi_0$ by means of the explicit formula
    \eqref{eq:HolomorphicShadow}. To this end, we choose $V \subseteq
    \LieAlg(G_\C)$ as an open neighbourhood of zero, on
    which the Lie exponential~$\exp_{G_\C}$ is injective. Hence, every $g \in
    U \coloneqq \exp_{G_\C} V$ is of the form $g = \exp_{G_\C} \xi$ for a unique~$\xi
    \in V$. Furthermore, by assumption, $\eta$ is locally injective and thus its tangent
    map is indeed invertible. Consequently, \eqref{eq:HolomorphicShadow} defines a
    function on $\exp_{G_\C} V$. Reading~\eqref{eq:HolomorphicShadow} as a
    coordinate representation within the holomorphic chart $\exp_{G_\C}$
    of $G_\C$ moreover establishes the holomorphicity $\phi_0$, where the
    $\C$-linearity of \eqref{eq:TangentMapExtension} is crucial. If now $g \in
    \eta^{-1}(U)$, then there exists a unique $\xi \in V$ with $\eta(g) = \exp_{G_\C}
    \xi$. Consequently,
    \begin{equation}
        \eta^*
        \phi_0
        \at[\Big]{g}
        =
        \phi_0
        \bigl(
            \exp_{G_\C} \xi
        \bigr)
        =
        \Taylor_\phi
        \bigl(
            \E,
            \underline{
                (T_\E \eta)^{-1}
                \xi
            }
        \bigr)
        =
        \phi
        \Bigl(
            \exp_G
            \bigl(
                (T_\E \eta)^{-1}
                \xi
            \bigr)
        \Bigr)
        =
        \phi(g)
    \end{equation}
    by virtue of the Lie-Taylor series from Theorem~\ref{thm:LieTaylor} and the
    commutativity of the diagram
     \begin{equation}
         \label{eq:LieCorrespondence}
        \begin{tikzcd}[column sep = huge]
            G
            \arrow[r, "\eta"]
            &G_\C \\
            \liealg{g}
            \arrow[u, "\exp_G"]
            \arrow[r, "T_\E \eta"]
            &\LieAlg(G_\C)
            \arrow[u, "\exp_{G_\C}"]
        \end{tikzcd}.
    \end{equation}
    Varying $g \in \eta^{-1}(U)$ establishes \eqref{eq:HolomorphicShadowPullback}.
    Finally, taking another look at \eqref{eq:LieDerivativeFlow}, the identity~
    \eqref{eq:HolomorphicShadowCoefficients} for the Lie-Taylor coefficients follows
    from the Lie correspondence \eqref{eq:LieCorrespondence} for any $\xi \in
    \Universal^\bullet(\liealg{g})$. By the left-invariant Cauchy-Riemann equations
    \eqref{eq:CauchyRiemann}, \eqref{eq:HolomorphicShadowCoefficients} then
    extends to arbitary
    $\xi \in \Universal^\bullet(\liealg{g}_\C)$.
\end{proof}

In the sequel, we call any $\phi_0 \in \Holomorphic(U)$ defined on an open domain
$U \subseteq G_\C$ containing the group unit a \emph{holomorphic shadow} of
$\phi$. The left-invariant identity principle from Corollary~\ref{cor:IdentityPrinciple}
implies that any two holomorphic shadows
\begin{equation}
    \phi_0
    \colon
    U
    \longrightarrow
    \C
    \qquad \textrm{and} \qquad
    \widehat{\phi}_0
    \colon
    \widehat{U}
    \longrightarrow
    \C
\end{equation}
of $\phi$ coincide on $(U \cap \widehat{U})_0$. By slight abuse of language, we thus
speak of \emph{the} holomorphic shadow $\phi_0$ of $\phi$ in the sequel. In
particular, we may always use the explicit formulas~\eqref{eq:HolomorphicShadow}
and~\eqref{eq:HolomorphicShadowCoefficients}. Note that, by
Proposition~\ref{prop:UniversalComplexification},
\ref{item:ComplexificationCovering} the additional assumption of local injectivity is
always fulfilled when working with the universal covering group. That is to say, every
entire function~$\psi \in \Entire(\Covering)$ casts a holomorphic shadow.

We proceed with proving that the entirety of the Lie-Taylor majorant of a
holomorphic function at one point already implies the entirety of its Lie-Taylor
majorant at all points of its domain. This will be instrumental for the analytic
continuation of the holomorphic shadow by means of successive applications of the
Lie-Taylor formula \eqref{eq:LieTaylor}.
\begin{lemma}
    \label{lem:MajorantTranslation}%
    Let $G$ be a connected Lie group, $U \subseteq G$ a domain and $\phi
    \in \Comega(U)$ a real-analytic such that the Lie-Taylor majorant
    $\Majorant_\phi(g_0,\argument)$ is entire for some $g_0 \in U$. Then the
    Lie-Taylor majorant $\Majorant(g,\argument)$ is entire for all $g \in U$.
\end{lemma}

The analogous statement for $\phi \in \Entire(G)$ was already established in
\cite[Theorem~4.17,~\textit{iii.)}]{heins.roth.waldmann:2023a} and our methods are
essentially identical. Working with globally defined functions, one may recast
Lemma~\ref{lem:MajorantTranslation} as an action of $G$ on $\Entire(G)$ by means
of translations.
\begin{proof}[Of Lemma~\ref{lem:MajorantTranslation}]
    Fix a basis $(\basis{e}_1, \ldots, \basis{e}_d)$ of the Lie algebra~$\liealg{g}$ with
    corresponding supremum norm~$\supnorm{\argument}$. As a first step, we
    establish the entirety of $\Majorant_\phi(g_0 \exp \xi, \argument)$ for sufficiently
    small $\xi \in \liealg{g}$. Recall that the Lie-Taylor formula \eqref{eq:LieTaylor}
    arises from the usual Taylor formula within an exponential chart, see again the
    proof of Proposition~\ref{prop:LieTaylorOneVariable}. In particular, this
    means that there exists some open neighbourhood $U_0 \subseteq G$ of $g_0$
    such that
    \begin{equation}
        \bigl(
            \Lie(\chi)
            \phi
        \bigr)(g_0 \exp \xi)
        =
        \sum_{k=0}^\infty
        \frac{1}{k!}
        \sum_{m_1, \ldots, m_k=1}^d
        \bigl(
            \Lie
            (
                \basis{e}_{m_1} \cdots \basis{e}_{m_k}
                \chi
            )
            \phi
        \bigr)(g_0)
        \cdot
        \xi^{m_1}
        \cdots
        \xi^{m_k}
    \end{equation}
     for all $\xi \in U_0$ and $\chi \in \Universal^\bullet(\liealg{g}_\C)$. Unwrapping
     the definition of the Lie-Taylor majorant thus leads to the estimate
    \begin{align}
        \Majorant_{\phi}
        \bigl(
            g_0 \exp \xi,
            \abs{z}
         \bigr)
        &=
        \sum_{n=0}^{\infty}
        \frac{\abs{z}^n}{n!}
        \sum_{\ell_1, \ldots, \ell_n = 1}^d
        \abs[\Big]
        {
            \bigl(
                \Lie(\basis{e}_{\ell_1} \cdots \basis{e}_{\ell_n})
                \phi
            \bigr)(g_0 \exp \xi )
        } \\
        &\le
        \sum_{n,k=0}^{\infty}
        \frac{\abs{z}^n}{n! \cdot k!}
        \sum_{\ell_1, \ldots, \ell_n = 1}^d
        \sum_{m_1, \ldots, m_k = 1}^d
        \supnorm{\xi}^k
        \cdot
        \abs[\Big]
        {
            \bigl(
            \Lie
            (
                \basis{e}_{m_1} \cdots \basis{e}_{m_k}
                \basis{e}_{\ell_1} \cdots \basis{e}_{\ell_n}
            )
            \phi
            \bigr)(g_0)
        } \\
        &=
        \sum_{n,k=0}^{\infty}
        \frac{\abs{z}^n \cdot \supnorm{\xi}^k}{n! \cdot k!}
        \sum_{\ell_1, \ldots, \ell_{n+k} = 1}^d
        \abs[\Big]
        {
            \bigl(
            \Lie
            (
            \basis{e}_{\ell_1} \cdots \basis{e}_{\ell_{n+k}}
            )
            \phi
            \bigr)(g_0)
        } \\
        &=
        \sum_{k=0}^{\infty}
        \frac{\supnorm{\xi}^k}{k!}
        \cdot
        \sum_{n=0}^\infty
        \frac{(n+k)!}{n!}
        \cdot
        c_{n+k}(\phi,g_0)
        \cdot
        \abs{z}^n \\
        &=
        \sum_{k=0}^{\infty}
        \frac{\supnorm{\xi}^k}{k!}
        \cdot
        \Majorant_\phi^{(k)}
        \bigl(
            g_0, \abs{z}
        \bigr) \\
        &=
        \Majorant_\phi
        \bigl(
            g_0, \abs{z} + \supnorm{\xi}
        \bigr)
    \end{align}
    for all $\xi \in U_0$ and $z \in \C$, where we have used the power series
    expansions of the entire function $\Majorant_\phi(g_0,\argument) \in
    \Holomorphic(\C)$ and its derivatives $\Majorant_\phi^{(k)}(g_0,\argument)$. We
    have thus established that the Taylor series of $\Majorant_\phi(g_0 \exp \xi)$
    converges absolutely and locally uniformly on $\C$. That is to say, it constitutes an
    entire function itself. Let now $g \in U$. As $U_0$ is a zero neighbourhood, its
    exponential image generates all of the connected Lie group $G$. Thus, we find Lie
    algebra elements $\xi_1,\ldots,\xi_n \in U_0$ with
    \begin{equation}
        g_0^{-1}
        g
        =
        \exp(\xi_1) \cdots \exp(\xi_n).
    \end{equation}
    That is to say, $g = g_0 \exp(\xi_1) \cdots \exp(\xi_n)$. Inductive application of
    what we have shown thus yields the entirety first of $\Majorant_\phi(g_0 \exp
    \xi_1)$, then of $\Majorant_\phi(g_0 \exp \xi_1 \exp \xi_2)$ and so on, completing
    the proof.
\end{proof}

The inequality we have just derived may be interpreted as a continuity estimate in
the following manner.
\begin{corollary}
    \label{cor:MajorantsAsSeminorms}%
    Let $G$ be a connected Lie group. Then
    \begin{equation}
        \Entire(G)
        \ni
        \phi
        \quad \mapsto \quad
        \Majorant_\phi(g,r)
        \in
        [0,\infty)
    \end{equation}
    is a continuous seminorm for all $g \in G$ and $r \ge 0$.
\end{corollary}

Moreover, combining Lemma~\ref{lem:MajorantTranslation} with
\eqref{eq:HolomorphicShadowCoefficients}, the following is
now immediate.
\begin{corollary}
    \label{cor:MajorantEntire}%
    Let $G$ be a connected Lie group and consider an entire function $\phi \in
    \Entire(G)$ with holomorphic shadow~$\phi_0 \in \Holomorphic(U)$. Then the
    Lie-Taylor majorant $\Majorant_{\phi_0}(g,\argument)$ is entire for all $g \in U$.
\end{corollary}

Preparing the final ingredient for our local extension result, we use the group
structure and a touch of Riemannian geometry to define an analogue of Steiner
chains on $G$, i.e. chains of circles such that each circle contains the center of its
successor.
\begin{lemma}[Steiner Chains]
    \label{lem:SteinerChains}%
    Let $G$ be a connected Lie group, $\gamma \colon [0,1] \longrightarrow G$ a
    path from $g_0$ to $g_1$ and $U \subseteq G$ an open neighbourhood of
    the group unit. Then there are
    \begin{equation}
        \label{eq:TimesGoUp}
        0
        =
        t_0
        <
        t_1
        <
        t_2
        < \cdots <
        t_k
        =
        1
    \end{equation}
    with
    \begin{equation}
        \label{eq:SteinerCenters}
        \gamma(t_{\ell+1})
        \in
        \gamma(t_\ell)
        U
        \qquad
        \textrm{for}
        \quad
        \ell = 1, \ldots, k-1
    \end{equation}
    and
    \begin{equation}
        \label{eq:SteinerCovering}
        \gamma\bigl([0,1]\bigr)
        \subseteq
        \bigcup_{\ell=0}^k
        \gamma(t_\ell)
        U.
    \end{equation}
\end{lemma}
\begin{proof}
    Let $\varrho$ be a left-invariant metric on $G$ as discussed in
    Proposition~\ref{prop:RiemannianLeftInvariant}. By
    \ref{item:RiemannianLengthTopology}, the metric topology reproduces the
    topology of $G$. Thus, there exists a closed metric ball
    $\Ball_r(\E)^\cl \subseteq U$. Moreover, the metric length of $\gamma$ is finite,
    and thus we may divide its trace into $k$ subarcs~$\gamma \at{[t_\ell, t_{\ell+1}]}$
    of length at most $r$ and with times adhering to \eqref{eq:TimesGoUp}. The
    subdividing points fulfil
     \begin{equation}
         \varrho
         \bigl(
            \gamma(t_{\ell+1}),
            \gamma(t_\ell)
         \bigr)
         \le
         r
         \qquad
         \textrm{for }
         \ell=0,\ldots,k-1.
     \end{equation}
     By left-invariance, we moreover have
     \begin{equation}
         \gamma(t_\ell)
         \Ball_r(\E)^\cl
         =
         \Ball_r
         \bigl(
            \gamma(t_\ell)
         \bigr)
         \qquad
         \textrm{for }
         \ell = 0,\ldots,k.
     \end{equation}
     Combining both observations proves \eqref{eq:SteinerCenters} and by virtue of
     \begin{equation}
         \bigcup_{\ell=0}^k
         \gamma(t_\ell)
         U
         \supseteq
         \bigcup_{\ell=0}^k
         \Ball_r
         \bigl(
            \gamma(t_\ell)
         \bigr)^\cl
         \supseteq
         \bigcup_{\ell=0}^k
         \gamma
         \at{[t_\ell, t_{\ell+1}]}
         =
         \gamma
         \bigl(
            [0,1]
         \bigr),
     \end{equation}
     we have also achieved \eqref{eq:SteinerCovering}. Here, we have used that each
     subarc is contained in the closed ball with radius equal to its length.
\end{proof}

With the help of Steiner chains, we may now continue the holomorphic shadow
$\phi_0$ analytically along any path in $G_\C$.
\begin{lemma}
    Let $G$ be a connected Lie group and let~$\phi \in \Entire(G)$. Then its
    holomorphic shadow $\phi_0$ may be analytically continued along any path
    $G_\C$ starting at the group unit.
\end{lemma}
\begin{proof}
    Let $\gamma \colon [0,1] \longrightarrow G_\C$ be a path from $\E$ to $g \in
    G_\C$ and $\phi_0 \in \Holomorphic(U)$ be a holomorphic shadow of $\phi$
    defined on a neighbourhood $U$ of the group unit. Moreover, let $V \subseteq
    \liealg{g}$ be a zero neighbourhood such that $\exp \colon V \longrightarrow
    \exp V$ constitutes a diffeomorphism. Without loss of generality, we may assume
    $\exp V = U$. Invoking Lemma~\ref{lem:SteinerChains}, we get
    \begin{equation}
        0
        =
        t_0
        <
         t_1
        < \cdots <
        t_k
        =
        1
    \end{equation}
    such that
    \begin{equation}
        \gamma
        \bigl(
            [0,1]
        \bigr)
        \subseteq
        \bigcup_{\ell=0}^k
        \gamma(t_\ell)
        U
        \qquad \textrm{and} \qquad
        \gamma(t_{\ell+1})
        \in
        \gamma(t_\ell)
        U
        \qquad
        \textrm{for }
        \ell =1,\ldots,k-1.
    \end{equation}
    By Corollary~\ref{cor:MajorantEntire}, we know that the Lie-Taylor majorant at
    $\gamma(t_1) \in \gamma(t_0) U = \exp V$ is entire. Hence, using that
    $\exp \colon V \longrightarrow \exp V$ is bijective, we may define a holomorphic
    mapping by
    \begin{equation}
        \phi_1
        \colon
        \gamma(t_1)
        \exp V
        \longrightarrow
        \C, \quad
        \phi_1
        \bigl(
            \gamma(t_1)
            \exp \xi
        \bigr)
        \coloneqq
        \Taylor_{\phi_0}
        \bigl(
           \gamma(t_1), \underline{\xi}
        \bigr).
    \end{equation}
    Note that $\exp V \cap \gamma(t_1)\exp V \neq \emptyset$, as $\gamma(t_1)$ is
    within by construction. Hence, the intersection even contains an open connected
    neighbourhood of $\gamma(t_1)$, on which $\phi_0$ and $\phi_1$
    coincide by virtue of the Lie-Taylor formula \eqref{eq:LieTaylor}. Invoking the
    identity principle, this implies $\phi_0 = \phi_1$ on
    \begin{equation}
        \bigl(
            U
            \cap
            \gamma(t_1)U
        \bigr)_0.
    \end{equation}
    As this intersection contains $\gamma(t_1)$ by construction, the Lie-Taylor
    majorants fulfil
    \begin{equation}
        \Majorant_{\phi_1}
        \bigl(
        \gamma(t_1),
        \argument
        \bigr)
        =
        \Majorant_{\phi_0}
        \bigl(
        \gamma(t_1),
        \argument
        \bigr)
        \in
        \Holomorphic(\C),
    \end{equation}
    which returns us to the original assumptions, now for $\phi_1$. We may thus
    iterate the procedure and define
    \begin{equation}
        \phi_2(\exp \xi)
        \coloneqq
        \Taylor_{\phi_1}
        \bigl(
            \gamma(t_2), \underline{\xi}
        \bigr)
    \end{equation}
    and so on. This yields the desired analytic continuation of $\phi_0$ along
    $\gamma$.
\end{proof}

Investing the Monodromy Theorem, as it can be found in the textbook
\cite[Sec.~10.3.4]{krantz:1999a}, yields the existence of a globally defined
holomorphic extensions under the additional assumption that $G_\C$ is simply
connected.
\begin{corollary}
    \label{cor:ExtensionSimplyConnected}%
    Let $G$ be a connected Lie group such that its universal complexification~$G_\C$
    is simply connected and $\phi \in \Entire(G)$. Then there exists a unique
    holomorphic function $\Phi \in \Holomorphic(G_\C)$ such that
    \begin{equation}
        \Phi
        \circ
        \eta
        =
        \phi.
    \end{equation}
\end{corollary}

The idea is now that this resolves the extension problem for the universal covering
group~$\Covering$, as its universal complexification is always simply connected by
Proposition~\ref{prop:UniversalComplexification},
\ref{item:ComplexificationCovering}. We may now put the discussion after
Theorem~\ref{thm:Extension} into practice to produce its proof.
\begin{proof}[Of Theorem~\ref{thm:Extension}]
    Let $\phi \in \Entire(G)$ and write $p \colon \Covering \longrightarrow G$ for the
    universal covering projection. By Proposition~\ref{prop:EntireFunctoriality}, the
    pullback $\psi \coloneqq p^* \phi$ is then an entire function on $\Covering$. As
    the universal covering group is simply connected by its very definition, we may
    invoke Corollary~\ref{cor:ExtensionSimplyConnected} to obtain a unique
    holomorphic mapping $\Psi \in \Holomorphic(\Covering_\C)$ such that
    \begin{equation}
        \Psi
        \circ
        \tilde{\eta}
        =
        \psi,
    \end{equation}
    where $\tilde{\eta} \colon \Covering \longrightarrow \Covering_\C$ denotes the
    universal complexification morphism for $\Covering$. We show that $\Psi$ is
    invariant under the action of $\tilde{\eta}(\pi_1(G))$. To this end, we first note that
    as a pullback along the covering projection, the mapping $\psi$ is
    $\pi_1(G)$-periodic, as $\pi_1(G) = \ker p$. Fix now $h \in \pi_1(G) \subseteq
    \tilde{G}$ and consider the auxiliary function
    \begin{equation}
        \Psi_h
        \colon
        \Covering_\C
        \longrightarrow
        \C, \quad
        \Psi_h(g)
        \coloneqq
        \Psi
        \bigl(
            g
            \tilde{\eta}(h)
        \bigr).
    \end{equation}
    Notice
    \begin{equation}
        \Psi_h
        \bigl(
            \tilde{\eta}(g)
        \bigr)
        =
        \Psi
        \bigl(
            \tilde{\eta}(gh)
        \bigr)
        =
        \psi(gh)
        =
        \psi(g)
        =
        \Psi
        \bigl(
            \tilde{\eta}(g)
        \bigr)
        \qquad
        \textrm{for all }
        g \in \tilde{G}.
    \end{equation}
    This implies $\Psi_h = \Psi$ by the identity principle from
    Corollary~\ref{cor:IdentityPrinciple}. Varying $h \in \pi_1(G)$, we get that $\Psi$
    is indeed invariant under the action of $\tilde{\eta}(\pi_1(G))$. Consequently, it
    factors through the quotient projection $\pi \colon \tilde{G}_\C \longrightarrow
    G_\C$, see again Proposition~\ref{prop:UniversalComplexification},
    \ref{item:ComplexificationAsQuotient}, yielding a function $\Phi \in
    \Holomorphic(G_\C)$. Recall now that $\eta \circ p = \pi \circ \tilde{\eta}$ by virtue
    of \eqref{eq:ComplexificationFromUniversalCovering} and thus
    \begin{equation*}
        \Phi
        \circ
        \eta
        \circ
        p
        =
        \Phi
        \circ
        \pi
        \circ
        \tilde{\eta}
        =
        \Psi
        \circ
        \tilde{\eta}
        =
        \psi
        =
        \phi
        \circ
        p.
    \end{equation*}
    Invoking the local bijectivity of $p$, this implies $\Phi \circ \eta = \phi$ as
    desired. Moreover, the uniqueness of~$\Phi$ is clear by
    Corollary~\ref{cor:IdentityPrinciple}. Finally,
    \eqref{eq:HolomorphicShadowCoefficients} is just \eqref{eq:ExtensionLieTaylor}
     again.
\end{proof}

\begin{remark}[Penney's Theorem]
    Let $G$ be a connected, solvable and simply connected Lie group. The paper
    \cite{penney:1974a} deals with holomorphic extensions of Lie group
    representations
    \begin{equation}
        \pi
        \colon
        G
        \longrightarrow
        \GL(V)
    \end{equation}
    on locally convex spaces $V$. Viewing the Lie-Taylor formula as the integration of
    $\Lie$ from \eqref{eq:LieDerivative} to the group representation given by right
    translation,~\cite[Lemma~2]{penney:1974a} provides another avenue of proof for
    Theorem~\ref{thm:Extension} if combined with our geometric considerations to
    pass to the holomorphic shadow as a preliminary step. Penney's technique is
    based on carefully counting the amount of lower order terms one picks up upon, if
    given $1 \le j_1, \ldots, j_n \le d$, one rewrites this as
    \begin{equation}
        \basis{e}_{j_1}
        \cdots
        \basis{e}_{j_n}
        =
        \basis{e}_{k_1}
        \cdots
        \basis{e}_{k_n}
        +
        \textrm{ lower order terms}
    \end{equation}
    with $1 \le k_1 \le k_2 \le \cdots \le k_n \le d$ within the universal enveloping
    algebra $\Universal^\bullet(\liealg{g})$. More on this representation theoretic
    point of view can be found in \cite[Rmk.~4.16]{heins.roth.waldmann:2023a} and
    \cite[Sec.~3.6 \& 3.7]{heins:2024a}. Be advised that the former refers to strongly
    entire vectors simply as entire.
\end{remark}

\section{Restriction: Cauchy Estimates}
\label{sec:Restriction}%

In this section, we always work with a complex Lie group $G$. Viewing it as a real Lie
group with real Lie algebra $\liealg{g}$, we may define $\Entire(G)$ as before. Taking
another look at Theorem~\ref{thm:Extension}, our goal is to establish the embedding
\begin{equation}
    \Holomorphic(G)
    \subseteq
    \Entire(G),
\end{equation}
which means we have to estimate the Lie-Taylor majorants by means of supremum
norms
\begin{equation}
    \norm{\phi}_K
    \coloneqq
    \max_{g \in K}
    \abs[\big]
    {\phi(z)}
\end{equation}
for compact subsets $K \subseteq G$, and vice versa. The latter is based on the
Lie-Taylor formula~\eqref{eq:LieTaylor}. Moreover, we achieve the former by deriving
a Lie-theoretic version of the Cauchy estimates, which bound each Lie-Taylor
coefficient individually, but ultimately in a uniform manner. We discuss two flavours
of essentially the same technique: one for operator norms and matrix Lie groups and
the other in full generality using the left-invariant Riemannian geometry we have
established within Proposition~\ref{prop:RiemannianLeftInvariant}.

Thus, assume for the moment that $G$ is linear and choose a finite
dimensional faithful representation. As both $\Holomorphic$ and~$\Entire$ are
functors, this choice indeed does not constitute a loss of generality. Working with
matrices $M \in \Mat_d(\C)$, we choose some auxiliary norm on~$\C^d$ and denote
the resulting operator norm simply by $\norm{M}$. By our assumption, it makes
sense to speak both of operator norms of group elements and of Lie algebra
elements. That is to say, there is a direct way to compare sizes of infinitesimal and
exponentiated elements. As the Lie exponential is simply given by the exponential
series for matrix Lie groups, we get the intuitive estimate
\begin{equation}
    \label{eq:OperatorNormExponential}
    \norm
    [\big]
    {\exp \xi}
    \le
    e^{\norm{\xi}}
    \qquad
    \textrm{for all }
    \xi \in \liealg{g},
\end{equation}
where $e$ denotes Euler's number. This may be thought of as a continuity estimate
for the Lie exponential much like \eqref{eq:RiemannianLengthLipschitz}. Keeping this
in mind, we arrive at the following first Lie-theoretic incarnation of the Cauchy
estimates.
\begin{lemma}[Lie-theoretic Cauchy estimates~I]
    \label{lem:CauchyEstimates}%
    Let $G$ be a complex connected matrix Lie group, $\phi \in \Holomorphic(G)$,
    $K_0 \subseteq \liealg{g}$ a compact neighbourhood of zero and $r > 0$. Then,
    writing
    \begin{equation}
        K_r
        \coloneqq
        \bigl\{
            g \in G
            \colon
            \norm{g}
            \le
            e^r
        \bigr\},
    \end{equation}
    we have
    \begin{equation}
        \label{eq:CauchyEstimates}
        \abs[\Big]
        {
            \bigl(
            \Lie(\xi_1 \cdots \xi_n)
            \phi
            \bigr)
            (g \exp \xi)
        }
        \le
        \norm{\phi}_{g\exp(K_0)K_r}
        \cdot
        \frac{n^n}{r^n}
    \end{equation}
    for all $g \in G$, $\xi \in K_0$, $n \in \N$ and $\xi_1, \ldots, \xi_n \in \liealg{g}$
    with $\norm{\xi_j} \le 1$ for $j=1, \ldots, n$.
\end{lemma}
\begin{proof}
    Fix $n \in \N$ and $\xi_1, \ldots, \xi_n \in \Ball_1(0)^\cl \subseteq \liealg{g}$.
    Evaluating \eqref{eq:LieDerivativeHigherOrder} for $k=1$ and $s = 0$ yields
    \begin{equation}
        \bigl(
            \Lie(\xi_j)
            \phi
        \bigr)
        (g \exp \xi)
        =
        \frac{\partial}{\partial z}
        \phi
        \bigl(
            g \exp \xi \cdot \exp(z \xi_j)
        \bigr)
        \at[\Big]{z = 0}
    \end{equation}
    for $j = 1, \ldots, n$, $g \in G$ and $\xi \in K_0$. As the Lie exponential induces
    a holomorphic atlas, the composition
    \begin{equation}
        \phi
        \circ
        \ell_{g \exp(\xi)}
        \circ
        \exp
        \colon
        \liealg{g} \longrightarrow \C
    \end{equation}
    is holomorphic. By induction, we get
    \begin{equation*}
        \Lie(\xi_1 \cdots \xi_n)
        \phi
        (g \exp \xi)
        =
        \frac{\partial}{\partial z^1}
        \cdots
        \frac{\partial}{\partial z^n}
        \phi
        \bigl(
            g \exp \xi \cdot \exp(z^n \xi_n) \cdots \exp(z^1 \xi_1)
        \bigr)
        \at[\Big]{z^1 = \ldots = z^n = 0},
    \end{equation*}
    where
    \begin{equation}
        \C^n
        \ni
        z
        \mapsto
        \phi
        \bigl(
        g
        \exp(\xi)
        \exp(z^n \xi_n)
        \cdots
        \exp(z^1 \xi_1)
        \bigr)
    \end{equation}
    is holomorphic. Notably, there are only first derivatives involved. Applying the
    usual Cauchy estimates for first derivatives with all radii equal to $r/n$ yields
    \begin{equation}
        \label{eq:CauchyEstimatesProof}
        \abs[\Big]
        {
            \bigl(
            \Lie(\xi_1 \cdots \xi_n)
            \phi
            \bigr)
            (g \exp \xi)
        }
        \le
        \frac{n^n}{r^n}
        \max_{\abs{z^1} = \cdots = \abs{z^n} = r/n}
        \abs[\Big]
        {
            \phi
            \bigl(
                g \exp \xi
                \cdot
                \exp(z^n \xi_n) \cdots \exp(z^1 \xi_1)
            \bigr)
        }.
    \end{equation}
    By submultiplicativity of the operator norm and
    \eqref{eq:OperatorNormExponential}, this implies
    \begin{align*}
        \norm[\big]
        {\exp(z^n \xi_n) \cdots \exp(z^1 \xi_1)}
        &\le
        \norm[\big]
        {\exp(z^n \xi_n)}
        \cdots
        \norm[\big]
        {\exp(z^1 \xi_1)} \\
        &\le
        \exp
        \bigl(
        \abs{z^n} \cdot \norm{\xi_n}
        \bigr)
        \cdots
        \exp
        \bigl(
        \abs{z^1} \cdot \norm{\xi_1}
        \bigr) \\
        &=
        \exp
        \bigl(
        \abs{z^n} \cdot \norm{\xi_n}
        +
        \cdots
        +
        \abs{z^n} \cdot \norm{\xi_1}
        \bigr) \\
        &\le
        \exp(r)
    \end{align*}
    for $\abs{z^1} = \cdots = \abs{z^n} = r/n$. That is to say, we have
    \begin{equation}
        \exp(z^n \xi_n)
        \cdots
        \exp(z^1 \xi_1)
        \in
        K_r
        \qquad
        \textrm{whenever}
        \quad
        \abs{z^1} = \ldots = \abs{z^n} = r/n.
    \end{equation}
    Consequently, \eqref{eq:CauchyEstimates} follows from
    \eqref{eq:CauchyEstimatesProof} by varying $z$ and $\xi$.
\end{proof}

Next, we replace the operator norms with the metric data induced by the choice of a
Riemannian metric $\mathbb{g}$ on $\liealg{g}$, as we have discussed in
Proposition~\ref{prop:RiemannianLeftInvariant}. This yields the following estimate,
which is structurally identical to \eqref{eq:CauchyEstimates}.
\begin{lemma}[Lie-theoretic Cauchy estimates II]
    \label{lem:CauchyEstimatesRiemannian}%
    Let $G$ be a complex connected Lie group, $\phi \in \Holomorphic(G)$ and
    $\mathbb{g} \in \liealg{g} \vee \liealg{g}$ be non-degenerate. We
    endow $\liealg{g}$ and $G$ with the corresponding norm $\norm{\argument}$
    and length $\varrho$, respectively. Moreover, let $K_0 \subseteq \liealg{g}$ be a
    compact neighbourhood of zero and $r > 0$. Then
    \begin{equation}
        \label{eq:CauchyEstimatesLipschitz}
        \abs[\Big]
        {
            \bigl(
            \Lie(\xi_1 \cdots \xi_n)
            \phi
            \bigr)
            (g \exp \xi)
        }
        \le
        \norm{\phi}_{g \exp(K_0) \Ball_r(\E)^\cl}
        \cdot
        \frac{n^n}{r^n}
    \end{equation}
    for all $g \in G$, $\xi \in K_0$, $n \in \N$ and $\xi_1, \ldots, \xi_n \in \liealg{g}$
    with $\norm{\xi_j} \le 1$ for $j=1, \ldots, n$.
\end{lemma}
\begin{proof}
    Fix $g \in G$, $\xi \in K_0$ and $\xi_1, \ldots, \xi_n \in \liealg{g}$. Taking another
    look at the proof of Lemma~\ref{lem:CauchyEstimates}, we did not use that $G$
    was a matrix Lie group in the derivation of \eqref{eq:CauchyEstimatesProof}.
    Combining \eqref{eq:RiemannianLengthMultiplicativity} with
    \eqref{eq:RiemannianLengthLipschitz}, we moreover have
    \begin{align}
        \varrho
        \bigl(
            \E,
            \exp(z^1 \xi_1)
            \cdots
            \exp(z^n \xi_n)
        \bigr)
        &\le
        \varrho(\E, z^1 \xi_1)
        + \cdots +
        \varrho(\E, z^n \xi_n) \\
        &\le
        \abs{z}^1
        \cdot
        \norm{\xi_1}
        + \cdots +
        \abs{z}^n
        \cdot
        \norm{\xi_n} \\
        &\le
        \frac{r}{n}
        \cdot
        1
        + \cdots +
        \frac{r}{n}
        \cdot
        1 \\
        &=
        r
    \end{align}
    for $\abs{z}^1 = \ldots = \abs{z}^n = r/n$. That is to say,
    \begin{equation}
        \exp(z^1 \xi_1)
        \cdots
        \exp(z^n \xi_n)
        \in
        \Ball_r(\E)^\cl
        \qquad
        \textrm{whenever}
        \quad
        \abs{z^1} = \ldots = \abs{z^n} = r/n.
    \end{equation}
    Varying $g$ and $\xi$ thus proves \eqref{eq:CauchyEstimatesLipschitz}.
\end{proof}

Having established the Lie-theoretic Cauchy estimates, we may return to
Theorem~\ref{thm:Restriction}.
\begin{proof}[Of Theorem~\ref{thm:Restriction}]
    Let $\mathbb{g} \in \liealg{g} \vee \liealg{g}$ be a non-degenerate two-form with
    accompanying norm~$\norm{\argument}$ on $\liealg{g}$ and distance $\varrho$
    on $G$. Consider $\phi \in \Holomorphic(G)$. Then $\phi \in \Comega(G)$, which
    means we only have to check the seminorm conditions. Let $(\basis{e}_1, \ldots,
    \basis{e}_d) \subseteq \liealg{g}$ be a basis such that $\norm{\basis{e}_n} = 1$ for
    $n=1,\ldots,d$. Applying Lemma~\ref{lem:CauchyEstimatesRiemannian} with
    constant $rd$, this yields
    \begin{equation}
        \seminorm{q}_{r}
        (\phi)
        =
        \sum_{n=0}^{\infty}
        \frac{r^n}{n!}
        \sum_{\alpha \in \N_d^n}
        \abs[\Big]
        {
            \bigl(
                \Lie(\alpha) \phi
            \bigr)(\E)
        }
        \le
        \norm{\phi}_{\Ball_r(\E)^\cl}
        \cdot
        \sum_{n=0}^{\infty}
        \frac{n!}{n^n}
        <
        \infty
    \end{equation}
    for any $r > 0$ by virtue of $\lim_{n \rightarrow \infty} \sqrt[n]{n!}/n = e^{-1}$. This
    moreover establishes the continuity of the $\Entire$-seminorms from
    \eqref{eq:EntireSeminorms} with respect to the topology of locally uniform
    convergence on $\Holomorphic(G)$. Conversely, let $K \coloneqq \Ball_1(0)^\cl
    \subseteq \liealg{g}$ be the closed unit ball corresponding to our norm. We cover
    $G$ by the translates $g \exp K$ by varying~$g$. The Lie-Taylor formula
    \eqref{eq:LieTaylor} yields then
    \begin{equation*}
        \norm[\big]
        {\phi}_{g\exp(K_0)}
        =
        \max_{\xi \in K_0}
        \abs[\big]
        {\phi(g \exp \xi)}
        \le
        \sum_{n=0}^{\infty}
        \frac{1}{n!}
        \sum_{\alpha \in \{1,\ldots,d\}^n}
        \abs[\Big]
        {
            \bigl(
                \Lie(\alpha) \phi
            \bigr)(g)
        }
        =
        \Majorant_\phi(g,\E)
    \end{equation*}
    for any $g \in G$. By Corollary~\ref{cor:MajorantsAsSeminorms},
    $\phi \mapsto \Majorant_{\phi}(g,\E)$ constitutes a continuous seminorm
    on~$\Entire(G)$, which completes the proof.
\end{proof}

We conclude the section with the sketch of an application to strict deformation
quantization, which is what the algebra~$\Entire(G)$ was originally conceived for
within \cite{heins.roth.waldmann:2023a}. Glossing over the motivation entirely, we
define the holomorphic polynomial algebra of $G_\C$ by
\begin{equation}
    \Holomorphic_\Pol(T^*G)
    \coloneqq
    \Holomorphic(G_\C)
    \tensor_\pi
    \Holomorphic_1(\liealg{g}_\C^*)
\end{equation}
where $\tensor_\pi$ denotes the projective tensor product as discussed e.g. in
\cite[§41]{koethe:1979a} and
\begin{equation}
    \Holomorphic_1(\liealg{g}_\C^*)
    \coloneqq
    \Holomorphic(\liealg{g}_\C^*)
    \cap
    \Entire_1(\liealg{g}_\C^*)
\end{equation}
in the sense of Remark~\ref{rem:PositiveR}. Note that this indeed makes sense, since
the dual $\liealg{g}_\C^*$ is a vector space and may thus be viewed as an abelian Lie
group. Theorem~\ref{thm:EntireVsHolomorphic} now allows us to identify the first
factor with $\Entire(G)$, which combined with
\cite[Thm~6.3]{heins.roth.waldmann:2023a} results in the following Theorem.
\begin{theorem}[Complex Standard Ordered Star Product]
    \label{thm:StarProductHolomorphicContinuity}%
    Let $G$ be a connected Lie group.
    \begin{theoremlist}
        \item The standard ordered star product
        \begin{equation}
            \star_\std
            \colon
            \Holomorphic_\Pol(T^*G)
            \tensor_\pi
            \Holomorphic_\Pol(T^*G)
            \longrightarrow
            \Holomorphic_\Pol(T^*G)
        \end{equation}
        is well defined and continuous.
        \item The mapping
        \begin{equation}
            \C
            \times
            \Holomorphic_\Pol(T^*G)
            \otimes
            \Holomorphic_\Pol(T^*G)
            \ni
            (\hbar, \Phi, \Psi)
            \mapsto
            \Phi
            \star_\std
            \Psi
            \in
            \Holomorphic_\Pol(T^*G)
        \end{equation}
        is Fréchet holomorphic.
    \end{theoremlist}
\end{theorem}

\begin{remark}[Real Forms]
    Not every complex Lie group arises as a universal complexification. By
    \cite[Ex.~5.1.24]{hilgert.neeb:2012a} there exists a complex Lie algebra $\liealg{h}$
    that may not be realized as $\liealg{g}_\C$ for any real Lie algebra $\liealg{g}$. Let
    $H$ be the unique connected and simply connected complex Lie group integrating
    $\liealg{h}$ and assume, for contradiction, that $G_\C \cong H$ for some Lie group
    $G$. Then $G$ is connected itself by the uniqueness in the universal property
    of~$G_\C$ and fulfils $\eta(G)_\C \cong H$ by Lemma~\ref{lem:ImageOfEta},
    \ref{item:ImageOfEtaComplexification}. But
    \begin{equation}
        \Lie
        \bigl(
        \eta(G)
        \bigr)_\C
        =
        \Lie(H)
    \end{equation}
    by \ref{item:ImageOfEtaDimension} of the same lemma, which is a contradiction.
\end{remark}

That being said, emulating the estimates from \cite{heins.roth.waldmann:2023a}, one
can extend Theorem~\ref{thm:StarProductHolomorphicContinuity} to arbitrary
complex Lie groups, not just the ones arising as universal complexifications. As
entering into the details would detract from the main points of
this paper, we refrain from putting this into practice. Instead, we refer to
\cite[Thm.~3.5.17]{heins:2024a} and conclude the section with
some remarks on further literature for the inclined reader.

The wonderful proceeding \cite{sternheimer:1998a} provides a somewhat dated, but
nevertheless insightful historical overview over the foundations and ideas of formal
deformation quantization. The central ideas behind strict deformation quantization
in the manner pursued above and further examples can be found within the
survey \cite{waldmann:2019a}. Our results may moreover be viewed as putting the
philosophy of holomorphic deformation quantization into practice, as it is discussed
within \cite[Ch.~1]{heins:2024a}.

\section{The Circle and the Special Linear Group}
\label{sec:Examples}%

Having established our main Theorems, we discuss some examples of what they
entail in action. We begin with the simplest situations beyond $\R^d$, which was
already treated without our more elaborate machinery within
\cite[Ex.~4.19]{heins.roth.waldmann:2023a}.
\begin{example}[Circle group]
    Let
    \begin{equation}
        G
        \coloneqq
        \liegroup{U}(1)
        \coloneqq
        \bigl\{
            z \in \C
            \colon
            \abs{z}
            =
            1
        \bigr\},
    \end{equation}
    whose universal complexification is given by the inclusion $\liegroup{U}(1)
    \hookrightarrow \C^\times \coloneqq \C \setminus \{0\}$. It is convenient to
    employ polar coordinates of the form $g = r \E^{\I t}$ with $r > 0$ and $t \in
    [0,2\pi)$, which makes it straightforward to check the universal property.

    If now $\Phi \in \Holomorphic(\C^\times)$, then we may develop $\Phi$ into a
    convergent Laurent expansion
    \begin{equation}
        \label{eq:LaurentExpansion}
        \Phi(z)
        =
        \sum_{n=-\infty}^\infty
        a_n
        z^n
        \qquad
        \textrm{for all }
        z \in \C^\times,
    \end{equation}
    where
    \begin{equation}
        \label{eq:LaurentCoefficients}
        a_n
        =
        \frac{1}{2\pi}
        \int_0^{2\pi}
        \Phi(\E^{\I t})
        \cdot
        \E^{- \I n t}
        \D t
        \qquad
        \textrm{for all }
        n \in \Z = -\N \cup \N_0.
    \end{equation}
    In particular, the Lie-Taylor coefficients of $\phi$ may be read off from
    \begin{equation}
        \label{eq:LaurentExp}
        \Phi
        \bigl(
            \E^{2 \pi \I t}
        \bigr)
        =
        \sum_{n=-\infty}^\infty
        a_n
        \E^{2 \pi \I n t}
    \end{equation}
    as
    \begin{equation}
        \label{eq:LaurentLieTaylor}
        \Lie(\xi)^k \Phi
        \at[\Big]{1}
        =
        \xi^k
        \cdot
        \Lie(1)^k \Phi
        \at[\Big]{1}
        =
        \xi^k
        \cdot
        \frac{\D^k}{\D t^k}
        \phi
        \bigl(
            \E^{2 \pi \I t}
        \bigr)
        \at[\Big]{t=0}
        =
        (2 \pi \I \xi)^k
        \sum_{n=-\infty}^\infty
        n^k a_n
    \end{equation}
    for all $\xi \in \C = \R_\C$. Notably, they are not simply given by the Laurent
    coefficients and as such quite complicated; the same is thus true for the Lie-Taylor
    expansion \eqref{eq:LieTaylor}. Nevertheless, the condition for $\Phi$ being entire
    becomes the finiteness of
    \begin{equation}
        \seminorm{q}_{r}(\Phi)
        =
        \sum_{k=0}^\infty
        \frac{(2 \pi r)^k}{k!}
        \abs[\bigg]
        {
            \sum_{n=-\infty}^{\infty}
            n^k
            a_n
        }
        \le
        \sum_{n=-\infty}^{\infty}
        \abs{a_n}
        \sum_{k=0}^\infty
        \frac{(2 \pi r n)^k}{k!}
        =
        \sum_{n=-\infty}^{\infty}
        \abs{a_n}
        \E^{2\pi r n}
        <
        \infty
    \end{equation}
    by absolute convergence of the series \eqref{eq:LaurentExp} for all $r \ge 0$.
    Hence, $\Phi \in \Entire(\C^\times)$, matching Theorem~\ref{thm:Restriction}. In
    particular, $\Phi \at{U(1)} \in \Entire(\liegroup{U}(1))$ as predicted by
    Theorem~\ref{thm:EntireVsHolomorphic}.

    Conversely, given $\phi \in \Entire(\liegroup{U(1)})$, we observe that the
    composition $\phi \circ \exp$ possesses a holomorphic extension to an entire
    function $F$ on $\C$. Using polar coordinates as before, we may then define
    \begin{equation}
        \Phi(r \E^{2\pi \I t})
        \coloneqq
        F
        \bigl(
            \log r
            \cdot
            \E^{2\pi \I t}
        \bigr).
    \end{equation}
    This idea is at the core of the earlier extension result
    \cite[Theorem~3.5.5]{heins:2024a} for compact connected Lie groups, which our
    Theorem~\ref{thm:Extension} supercedes. Its proof was based on the existence of
    a polar decomposition for any
    compact Lie group.

    Notably, the Laurent coefficients from~\eqref{eq:LaurentCoefficients} are
    well-defined for $\phi$ and thus one may attempt to define an holomorphic
    extension by means of \eqref{eq:LaurentExpansion} directly. It would be
    interesting to make this precise.
\end{example}

The next example showcases the pathologies the non-injectivitity of $\eta$ may
result in.
\begin{example}[Automatic periodicity]
    \label{ex:AutomaticPeriodicity}%
    We consider the special linear group
    \begin{equation}
        G
        \coloneqq
        \liegroup{SL}_2(\R)
        =
        \bigl\{
        M \in \Mat_2(\R)
        \colon
        \det(M) = 1
        \bigr\},
    \end{equation}
    which is connected, but not simply connected. Its universal covering group is
    $\tilde{G} = \widetilde{\liegroup{SL}}_2(\R)$, which is a nonlinear group
    by \cite[Example~9.5.18]{hilgert.neeb:2012a}, i.e. does not possess any faithful
    finite-dimensional representations. We denote the corresponding universal
    covering projection by
    \begin{equation}
        p
        \colon
        \widetilde{\liegroup{SL}}_2(\R)
        \longrightarrow
        \liegroup{SL}_2(\R).
    \end{equation}
    The Lie algebras of both groups are given by the traceless matrices
    \begin{equation}
        \liealg{sl}_2(\R)
        \coloneqq
        \bigl\{
        M
        \in
        \Mat_2(\R)
        \colon
        \tr(M)
        =
        0
        \bigr\},
    \end{equation}
    whose complexification is simply
    \begin{equation}
        \liealg{sl}_2(\R)_\C
        \cong
        \bigl\{
        M
        \in
        \Mat_2(\C)
        \colon
        \tr(M)
        =
        0
        \bigr\}
        \eqqcolon
        \liealg{sl}_2(\C).
    \end{equation}
    Consequently, by the construction within
    Proposition~\ref{prop:UniversalComplexification}, we have
    \begin{equation}
        \widetilde{\liegroup{SL}}_2(\R)_\C
        \cong
        \liegroup{SL}_2(\C)
        \coloneqq
        \bigl\{
            M
            \in
            \Mat_2(\C)
            \colon
            \det(M)
            =
            1
        \bigr\},
    \end{equation}
    as $\liegroup{SL}_2(\C)$ is simply connected by
    \cite[Exercise~15.1.1]{hilgert.neeb:2012a} with Lie algebra $\liealg{sl}_2(\C)$. Next,
    we prove that $\tilde{\eta} = \iota \circ p$, where $\iota \colon \liegroup{SL}_2(\R)
    \hookrightarrow \liegroup{SL}_2(\C)$ simply reinterprets real matrices as complex
    ones. Indeed, by construction its tangent map is given by
    \begin{equation}
        T_\E
        \tilde{\eta}
        =
        T_\E
        \iota
        =
        T_\E
        \iota
        \circ
        T_\E p ,
    \end{equation}
    where $T_\E \iota$ is the inclusion $\liealg{sl}_2(\R) \hookrightarrow
    \liealg{sl}_2(\C)$ and by connectedness the claimed form of~$\tilde{\eta}$ follows.
    Notably, $\tilde{\eta}$ is not injective, but nevertheless locally injective.
    Moreover, we get
    \begin{equation}
        \tilde{\eta}
        \bigl(
            \pi_1(G)
        \bigr)
        =
        \{\E\}.
    \end{equation}
    Consequently,~\eqref{eq:ComplexificationFromUniversalCovering} yields that
    $\liegroup{SL}_2(\R)_\C = \liegroup{SL}_2(\C)$ with $\eta \coloneqq \iota$.

    This example is remarkable in at least two regards: firstly, we see that non
    isomorphic Lie groups may have isomorphic universal complexifications. Secondly,
    \begin{equation}
        \pi_1(G_\C)
        =
        \pi_1
        \bigl(
            \liegroup{SL}_2(\C)
        \bigr)
        =
        \{\E\}
        \subsetneq
        \Z
        =
        \pi_1
        \bigl(
            \liegroup{SL}_2(\R)
        \bigr)
        =
        \pi_1(G).
    \end{equation}
    Let now $\phi \in \Entire(\widetilde{\liegroup{SL}}_2(\R))$. By
    Theorem~\ref{thm:Extension}, there exists a unique holomorphic extension
    \begin{equation}
        \Phi
        \in
        \Holomorphic
        \bigl(
        \liegroup{SL}_2(\C)
        \bigr)
    \end{equation}
    such that $\Phi \circ p = \phi$. This means that $\Phi$ factors through $p$ and is
    thus necessarily periodic with respect to the fundamental group
    $\pi_1(\liegroup{SL}_2(\R))$. Consequently, the algebra
    $\Entire(\widetilde{\liegroup{SL}}_2(\R))$ does not separate points.
\end{example}

\bibliographystyle{nchairx}
\phantomsection
\addcontentsline{toc}{section}{Bibliography}
\bibliography{Auxiliary/Entire}

@book{duistermaat.kolk:2000a,
    author =       {Duistermaat, J. J. and Kolk, J. A. C.},
    title =        {Lie Groups},
    publisher =    {Springer-Verlag},
    year =         {2000},
    address =      {Berlin, Heidelberg, New York},
}

@book{forstneric:2011a,
    author =       {Franc Forstneri\u{c}},
    title =        {Stein manifolds and holomorphic mappings},
    subtitle =     {The homotopy principle in complex analysis},
    year =         {2011},
    doi =          {10.1007/978-3-642-22250-4},
    publisher =    {Springer-Verlag Berlin Heidelberg},
}

@phdthesis{heins:2024a,
    author  = {Michael Heins},
    title   = {A Holomorphic perspective of Strict Deformation Quantization},
    school  = {Julius-Maximilians-Universit\"at W\"urzburg},
    year    = {2024},
    month   = {October},
    eprint  = {2504.12862},
    archivePrefix = {arXiv},
    primaryClass = {math.CV},
    url = {https://arxiv.org/abs/2504.12862},
}

@article{heins.roth.waldmann:2023a,
    author =       {Heins, M. and Roth, O. and Waldmann, S.},
    title =        {Convergent Star Products on Cotangent Bundles of Lie
    Groups},
    journal =      {Math. Annalen},
    year =         {2023},
    pages =        {151--206},
    volume =       {386},
    mrnumber =     {},
}

@book{helgason:2001a,
    author =       {Helgason, S.},
    title =        {Differential geometry, {L}ie groups, and symmetric
    spaces},
    series =       {Graduate Studies in Mathematics},
    volume =       {34},
    note =         {Corrected reprint of the 1978 original},
    publisher =    {American Mathematical Society},
    address =      {Providence, RI},
    year =         {2001},
    mrnumber =     {1834454},
    url =          {https://doi.org/10.1090/gsm/034},
}

@book{hilgert.neeb:2012a,
    author =       {Hilgert, J. and Neeb, K.-H.},
    title =        {Structure and geometry of {L}ie groups},
    series =       {Springer Monographs in Mathematics},
    publisher =    {Springer-Verlag},
    address =      {Heidelberg, New York},
    year =         {2012},
    mrnumber =     {3025417},
    doi =          {10.1007/978-0-387-84794-8},
    url =          {http://dx.doi.org/10.1007/978-0-387-84794-8},
}

@article{hochschild:1966a,
    author =       {Hochschild, G.},
    title =        {Complexification of Real Analytic Groups},
    year =         {1966},
    journal =      {Trans. Am. Math. Soc.},
    volume =       {125},
    pages =        {406-413},
    number =       {3},
}

@book{koethe:1979a,
    author =       {K{\"{o}}the, G.},
    title =        {Topological Vector Spaces {II}},
    publisher =    {Springer-Verlag},
    address =      {Heidelberg, Berlin, New York},
    year =         {1979},
    series =       {Grundlehren der mathematischen {W}issenschaft},
    number =       {237},
}

@book{krantz:1999a,
    author =       {Krantz, Steven G.},
    title =        {Handbook of complex variables},
    publisher =    {Birkh\"auser Boston, Inc., Boston, MA},
    year =         {1999},
    pages =        {xxiv+290},
    isbn =         {0-8176-4011-8},
    mrclass =      {30-00 (30-01)},
    mrnumber =     {1738432},
    mrreviewer =   {D.\ Mitrovi\'c},
    doi =          {10.1007/978-1-4612-1588-2},
    url =          {https://doi.org/10.1007/978-1-4612-1588-2},
}

@book{lee:2012a,
    title =        {Introduction to Smooth Manifolds},
    edition =      {2},
    author =       {Lee, J. M.},
    series =       {Graduate Texts in Mathematics},
    volume =       {218},
    year =         {2012},
    publisher =    {Springer New York},
}

@book{lelong.gruman:1986a,
    author =       {Lelong, Pierre and Gruman, Lawrence},
    title =        {Entire functions of several complex variables},
    series =       {Grundlehren der mathematischen Wissenschaften},
    volume =       {282},
    publisher =    {Springer-Verlag, Berlin},
    year =         {1986},
    pages =        {xii+270},
    isbn =         {3-540-15296-2},
    mrclass =      {32-02 (32A15 32Hxx)},
    mrnumber =     {837659},
    mrreviewer =   {L.\ I.\ Ronkin},
    doi =          {10.1007/978-3-642-70344-7},
    url =          {https://doi.org/10.1007/978-3-642-70344-7},
}

@book{osborne:2014a,
    author =       {Osborne, M. S.},
    title =        {Locally convex spaces},
    series =       {Graduate Texts in Mathematics},
    volume =       {269},
    publisher =    {Springer-Verlag},
    address =      {Cham, Heidelberg, New York},
    year =         {2014},
    mrnumber =     {3154940},
}

@article {penney:1974a,
    author =       {Penney, Richard},
    title =        {Entire vectors and holomorphic extension of
    representations.  {I}, {II}},
    journal =      {Trans. Amer. Math. Soc.},
    volume =       {198},
    year =         {1974},
    pages =        {107--121},
    issn =         {0002-9947,1088-6850},
    mrclass =      {22E45},
    mrnumber =     {364556},
    doi =          {10.2307/1996749},
    url =          {https://doi.org/10.2307/1996749},
}

@inproceedings{sternheimer:1998a,
    title={Deformation quantization: Twenty years after},
    ISSN={0094-243X},
    url={http://dx.doi.org/10.1063/1.57093},
    DOI={10.1063/1.57093},
    booktitle={AIP Conference Proceedings},
    publisher={AIP},
    author={Sternheimer, Daniel},
    year={1998},
    pages={107–145},
}

@book{tu:2017a,
    author =       {Tu, Loring W.},
    title =        {Differential geometry},
    series =       {Graduate Texts in Mathematics},
    volume =       {275},
    note =         {Connections, curvature, and characteristic classes},
    publisher =    {Springer-Verlag},
    address =      {Berlin, Heidelberg, Cham},
    year =         {2017},
    mrnumber =     {3585539},
    doi =          {10.1007/978-3-319-55084-8},
    url =          {https://doi.org/10.1007/978-3-319-55084-8},
}

@article{waldmann:2019a,
    author =       {Waldmann, S.},
    title =        {Convergence of Star Products: From Examples to a
    General Framework},
    journal =      {EMS Surv. Math. Sci.},
    volume =       {6},
    year =         {2019},
    pages =        {1--31},
}

\end{document}